\documentclass[12pt]{amsart}
\usepackage{amsmath,amssymb,amsthm}
\usepackage[all]{xy}
\usepackage{txfonts}
\theoremstyle{plain}
	\newtheorem{thm}{Theorem}[section]
	\newtheorem{prp}[thm]{Proposition}
	\newtheorem{lem}[thm]{Lemma}
	\newtheorem{cor}[thm]{Corollary}
\theoremstyle{definition}
	\newtheorem{dfn}[thm]{Definition}
	\newtheorem{ex}[thm]{Example}
\theoremstyle{remark}
	\newtheorem{rem}[thm]{Remark}

\newcommand{ \Z}{\mathbb{Z}}

\newcommand{ \Map}{{\rm Map}\hspace{0.1em}}

\newcommand{ \ad}{\mathrm{ad}\hspace{0.1em}}

\newcommand{ \SO}{\mathrm{SO}}
\newcommand{ \SU}{\mathrm{SU}}
\newcommand{ \Sp}{\mathrm{Sp}}

\begin{document}
\title{Homotopy pullback of $A_n$-spaces and its applications to $A_n$-types of gauge groups}
\author{Mitsunobu Tsutaya}
\address{Department of Mathematics, Kyoto University, Kyoto 606-8502, Japan}
\email{tsutaya@math.kyoto-u.ac.jp}
\keywords{homotopy pullback, $A_n$-space, $A_n$-map, fiberwise $A_n$-space,  gauge group, mapping space}
\subjclass[2010]{54C35 (primary), 18D50, 55P45, 55R10 (secondary)}
\begin{abstract}
We construct the homotopy pullback of $A_n$-spaces and show some universal property of it.
As the first application, we review the Zabrodsky's result which states that for each prime $p$, there is a finite CW complex which admits an $A_{p-1}$-form but no $A_p$-form.
As the second application, we investigate $A_n$-types of gauge groups.
In particular, we give a new result on $A_n$-types of the gauge groups of principal $\SU(2)$-bundles over $S^4$, which is a complete classification when they are localized away from 2.
\end{abstract}
\date{}
\maketitle

\section{Introduction}
The concept of \textit{$A_n$-space} was first introduced by Stasheff in \cite{Sta63a} as an $H$-space with higher homotopy associativity.
In \cite{Sta63b}, he also defined morphism between them, called \textit{$A_n$-homomorphism}, but, as pointed out by himself, it is too restrictive class.
Later, more general morphism between them, called \textit{$A_n$-map}, were formulated by Boardaman and Vogt \cite{BV73} and by Iwase \cite{Iwa83} (Iwase's construction can also be found in \cite{IM89}).

Before their formulation of $A_n$-map, to construct mixing of $A_n$-types, ``homotopy pullback of $A_n$-maps'' was considered by Zabrodsky \cite{Zab70} and by Mimura, Nishida and Toda \cite{MNT71} in certain sense.
Iwase and Mimura \cite{IM89} gave an outline of the proof of that the homotopy pullback of $A_n$-maps becomes an $A_n$-space, using the $A_n$-structures.

The first aim of this paper is to construct the homotopy pullback of $A_n$-spaces by giving an $A_n$-form on the usual homotopy pullback of spaces.
As pointed above, this is not a new result but a new realization.
We will consider the homotopy pullback of $A_n$-spaces by $A_n$-homomorphisms first (Section \ref{pullback1}).
After that, we consider the homotopy pullback by general $A_n$-maps (Section \ref{pullback2}).
As an application, we revisit the result of Zabrodsky \cite{Zab70}.

We will also investigate $A_n$-types of gauge groups.
For a principal $G$-bundle $P$, the \textit{gauge group} $\mathcal{G}(P)$ is the topological group consisting of self-isomorphisms on $P$. 
If principal $G$-bundles $P$ and $P'$ are isomorphic, then the gauge groups $\mathcal{G}(P)$ and $\mathcal{G}(P')$ are isomorphic as topological groups.
Considering the converse of this statement, if we replace `isomorphic as topological groups' by `homotopy equivalent as topological spaces', then it becomes much far from true.
This was first pointed out by Kono's result \cite{Kon91}.
As is well known, the principal $\SU(2)$-bundles over the four dimensional sphere $S^4$ is classified by the homotopy group $\pi _4(B\SU(2))\cong \mathbb{Z}$ of the classifying space.
Kono's result says that there are only 6 homotopy types of the gauge groups of them.

More generally, for a compact connected Lie group $G$ and a finite CW complex $B$, Crabb and Sutherland \cite{CS00} showed that the number of homotopy types of the gauge groups of principal $G$-bundles over $B$ is finite.
This finiteness holds not only for homotopy types but also for $H$-types (i.e. equivalence classes as $H$-spaces).

Kishimoto and Kono \cite{KK10} gave the condition for that the adjoint group bundle $\ad P$ (see Section \ref{gauge}) is trivial as a fiberwise $A_n$-space.
We note that the space of sections $\varGamma (\ad P)$ of $\ad P$ is naturally isomorphic to the gauge group $\mathcal{G}(P)$.

The author \cite{Tsu12a} generalized the result of Crabb and Sutherland for `$A_n$-equivalence types' with $n<\infty$ (the case when $n=2$ had been already known by them).
He \cite{Tsu12a}, \cite{Tsu12b} also considered the classification of $A_n$-types of gauge groups of principal $\SU (2)$-bundles over $S^4$.

Though there are many complete results on classifications of gauge groups, for example, \cite{Kon91}, \cite{HK06}, \cite{HK07}, \cite{KKKT07}, \cite{HKK08}, most of them were shown by observing the order of the Samelson product and using \cite[Lemma 3.2]{HK06}, which states that in some good situation, two gauge groups are homotopy equivalent if their $p$-localizations are homotopy equivalent for each prime $p$.
The second aim of this paper is to generalize this statement for $A_n$-types of gauge groups as follows (Section \ref{gauge}).

\begin{thm}\label{thm}
Let $G$ be a compact connected Lie group, of which the rationalization $G_{(0)}$ is homotopy equivalent to the product $S^{2n_1-1}_{(0)}\times \cdots \times S^{2n_{\ell }-1}_{(0)}$ of rationalized spheres.
Fix a map $\epsilon :S^{r-1}\to G$ with $r\geq 2n_\ell $.
For an integer $k\in \mathbb{Z}$, denote the principal $G$-bundle over $S^r$ with classifying map $k\epsilon :S^{r-1}\to G$ by $P_k$.
Take an integer $N\in \mathbb{Z}$ such that the adjoint group bundle $\ad P_N$ is $A_n$-trivial.
Then, the gauge groups $\mathcal{G}(P_k)$ and $\mathcal{G}(P_{k'})$ are $A_n$-equivalent if $(N,k)=(N,k')$, where $(a,b)$ represents the greatest common divisor of $a$ and $b$.
\end{thm}

We note that when $n=1$, it is easy to see that $\ad P_N$ is fiberwise homotopy equivalent to the trivial bundle if and only if the Samelson product $\langle \epsilon ,\mathrm{id}_G\rangle \in [S^{r-1}\wedge G,G]_0$ is annihilated by $N$.
To generalize the homotopy version to the above $A_n$-version, in some sense, we need to control the $A_n$-form of the homotopy pullback.

Using the above theorem, we will show the following new result (Section \ref{section_classification}).

\begin{thm}\label{SU(2)}
Denote the principal $\SU(2)$-bundle over $S^4$ with second Chern number $k\in \mathbb{Z}$ by $P_k$.
Then, for each positive integer $n$, the gauge groups $\mathcal{G}(P_k)$ and $\mathcal{G}(P_{k'})$ are $A_n$-equivalent if $\min\{2n,v_2(k)\}=\min\{2n,v_2(k')\}$ and $\min\{[2n/(p-1)],v_p(k)\}=\min\{[2n/(p-1)],v_p(k')\}$ for any odd prime $p$, where $v_p(k)$ denotes the $p$-exponent of $k$ and $[m]$ the maximum integer not greater than $m$.
Moreover, if $v_2(k)\leq1$, the converse is also true.
\end{thm}

In order to use various exponential laws, \textbf{we will always work in the category of compactly generated spaces and require pointed spaces have the homotopy extension property of the basepoints}.

The author would like to express his gratitude to Professors Norio Iwase, Daisuke Kishimoto and Akira Kono for giving valuable comments on this work.

\section{$A_n$-spaces and $A_n$-maps}\label{anspaces}
Let us denote the \textit{$i$-th associahedron} by $\mathcal{K}_i$ $(i=2,3,\ldots)$ and the \textit{$i$-th multiplihedron} by $\mathcal{J}_i$ $(i=1,2,\ldots)$.
There are boundary maps and degeneracy maps of them:
\begin{align*}
\partial _k(r,s)&:\mathcal{K}_r\times \mathcal{K}_s\to \mathcal{K}_{r+s-1} &&(1\leq k\leq r),\\
s_k&:\mathcal{K}_i\to \mathcal{K}_{i-1} &&(1\leq k\leq i),\\
\delta _k(r,s)&:\mathcal{J}_r\times \mathcal{K}_s\to \mathcal{J}_{r+s-1} &&(1\leq k\leq r),\\
\delta (t,r_1,\ldots ,r_t)&:\mathcal{K}_t\times \mathcal{J}_{r_1}\times \cdots \times \mathcal{J}_{r_t}\to \mathcal{J}_{r_1+\cdots +r_t},\\
d_k&:\mathcal{J}_i\to \mathcal{J}_{i-1} &&(1\leq k\leq i),
\end{align*}
which satisfy some relations.
Define the boundaries $\partial \mathcal{K}_i\subset \mathcal{K}_i$ and $\partial \mathcal{J}_i\subset \mathcal{J}_i$ as the unions of the images of boundary maps.
Then there are homeomorphisms $(\mathcal{K}_i,\partial \mathcal{K}_i)\cong (D^{i-2},S^{i-3})$ and $(\mathcal{J}_i,\partial \mathcal{J}_i)\cong (D^{i-1},S^{i-2})$, where $D^i$ is the $i$-dimensional disk and $S^{i-1}\subset D^i$ is its boundary sphere.
For details, see \cite{Sta63a} and \cite{IM89}.

Now, we recall the definitions of $A_n$-spaces and $A_n$-maps.
\begin{dfn}
Let $X$ be a pointed space.
A family of maps $\{m_i:\mathcal{K}_i\times X^i\to X\}_{i=2}^n$ is said to be an \textit{$A_n$-form} on $X$ if the following conditions are satisfied:
\begin{enumerate}
\item $m_2(x,*)=m_2(*,x)=x$
\item
$m_{r+s-1}(\partial _k(r,s)(\rho,\sigma);x_1,\ldots,x_{r+s-1})$\\
$=m_r(\rho;x_1,\ldots,x_{k-1},m_s(\sigma;x_k,\ldots ,x_{k+s-1}),x_{k+s},\ldots,x_{r+s-1})$
\item
$m_{i-1}(s_k(\rho );x_1,\ldots ,x_{i-1})=m_i(\rho ;x_1,\ldots ,x_{k-1},*,x_k,\ldots ,x_{i-1})$
\end{enumerate}
The pair $(X,\{m_i\})$ of a pointed space and its $A_n$-form is called an \textit{$A_n$-space}.
\end{dfn}
\begin{ex}
Let $G$ be a topological monoid.
Then $G$ is an $A_\infty$-space with $A_\infty$-form $\{m_i\}$ defined by
\[
m_i(\rho ; x_1,\ldots ,x_i)=x_1\cdots x_i.
\]
We call $\{m_i\}$ the \textit{canonical $A_\infty$-form} of $G$.
Unless otherwise stated, we regard a topological monoid as an $A_\infty$-space with the canonical $A_\infty$-form.
\end{ex}
\begin{dfn}
Let $(X,\{m_i\})$ and $(X',\{m_i'\})$ be $A_n$-spaces and $f:X\to X'$ a pointed map.
A family of maps $\{f_i:\mathcal{J}_i\times X^i\to X'\}_{i=1}^n$ is said to be an \textit{$A_n$-form} on $f$ if the following conditions are satisfied:
\begin{enumerate}
\item
$f_1=f$,
\item
$f_{r+s-1}(\delta _k(r,s)(\rho ,\sigma );x_1,\ldots ,x_{r+s-1})$\\
$=f_r(\rho ;x_1,\ldots,x_{k-1},m_s(\sigma;x_k,\ldots ,x_{k+s-1}),x_{k+s},\ldots,x_{r+s-1})$
\item
$f_{r_1+\cdots +r_t}(\delta (t,r_1,\ldots ,r_t)(\tau ,\rho _1,\ldots ,\rho _t);x_1,\ldots ,x_{r_1+\cdots +r_t})$\\
$=m_t'(\tau ;f_{r_1}(\rho _1;x_1,\ldots ,x_{r_1}),\ldots ,f_{r_t}(\rho _t;x_{r_1+\cdots +r_{t-1}+1},\ldots ,x_{r_1+\cdots +r_t}))$
\item
$f_{i-1}(d_k(\rho );x_1,\ldots ,x_{i-1})=f_i(\rho ;x_1,\ldots ,x_{k-1},*,x_k,\ldots ,x_{i-1})$
\end{enumerate}
The pair $(f,\{f_i\})$ of a pointed map and its $A_n$-form is called an \textit{$A_n$-map}.
For $A_n$-maps $(f,\{f_i\}),(f',\{f_i'\}):(X,\{m_i\})\to (X',\{m_i'\})$, a \textit{homotopy} between them is a continuous family of $A_n$-maps $(F:I\times X\to X',\{F_i:I\times \mathcal{J}_i\times X^i\to X'\})$ parametrized by the unit interval $I=[0,1]$ such that $F_i(0,\rho;x_1,\ldots ,x_i)=f_i(\rho ;x_1,\ldots ,x_i)$ and $F_i(1,\rho;x_1,\ldots ,x_i)=f_i'(\rho ;x_1,\ldots ,x_i)$.
\end{dfn}
\begin{rem}
In some literature, an $A_n$-map is defined as a map which admits some $A_n$-form.
But we always consider $A_n$-maps together with $A_n$-forms.
\end{rem}
\begin{rem}
In the author's previous work \cite{Tsu12a}, $A_n$-forms are not required to satisfy the condition concerning degeneracies.
So, we will pay attention to quote a result in that work if necessary.
\end{rem}
If $(f,\{f_i\}):(X,\{m_i\})\to (X',\{m_i'\})$ is an $A_n$-map and a pointed map $f':X\to X'$ is pointed homotopic to $f$, then, by homotopy extension property, $f'$ also admits an $A_n$-form $\{f_i'\}$ such that $(f',\{f_i'\})$ is homotopic to $(f,\{f_i\})$.

The definition of $A_n$-maps is complicated and it is difficult to treat in general.
So we often consider some class of $A_n$-maps defined more easily.
\begin{dfn}
Let $(X,\{m_i\})$ and $(X',\{m_i'\})$ be $A_n$-spaces.
A pointed map $f:X\to X'$ is said to be an \textit{$A_n$-homomorphism} if $f$ satisfies
\[
f(m_i(\rho ;x_1,\ldots ,x_i))=m_i'(\rho ;x_1,\ldots ,x_i).
\]
\end{dfn}
Now we see that an $A_n$-homomorphism admits an $A_n$-form.
Iwase and Mimura \cite{IM89} constructed maps $\pi_i:\mathcal{J}_i\to \mathcal{K}_i$ $(i=2,3,\ldots )$ such that
\begin{enumerate}
\item
$\pi _i\circ \delta _1(1,i)=\pi _i\circ \delta (i,1,\ldots ,1)=\mathrm{id}_{\mathcal{K}_i},$
\item
$\pi _{r+s-1}\circ \delta _k(r,s)=\partial _k(r,s)\circ (\pi _r\times\mathrm{id}_{\mathcal{K}_s}),$
\item
$\pi _{r_1+\cdots +r_t}\circ \delta (t,r_1,\ldots ,r_t)$\\
$\! =\underline{\! \partial _{r_1+\cdots +r_{t-1}+1}(r_1\! +\! \cdots \! +\! r_{t-1}\! +\! 1,r_t)\! \circ \cdots \circ \! (\partial _{r_1+1}(r_1\! +\! t\! -\! 1,r_2)\! \times \! \mathrm{id}_{\mathcal{K}_{r_3}\times \cdots \times \mathcal{K}_{r_t}})\! \circ \! (\partial _{1}(t,r_1)\! \times \! {\rm id}_{\mathcal{K}_{r_2}\times \cdots \times \mathcal{K}_{r_t}})}\! \circ \! ({\rm id}_{\mathcal{K}_t}\! \times \! \pi _{r_1}\! \times \! \cdots \! \times \! \pi _{r_t})$,
\item
$\pi _i\circ d_k=s_k\circ \pi _i$.
\end{enumerate}
Denote the underlined composite in the third equation by $D$.
Then $D$ has the property that
\[
m_{r_1+\cdots +r_t}(D(\tau ,\rho _1,\cdots ,\rho _t);x_1,\cdots ,x_{r_1+\cdots +r_t})=m_t(\tau ;m_{r_1}(\rho _1;x_1,\cdots ,x_{r_1}),\cdots ,m_{r_t}(\rho _t;\cdots ,x_{r_1+\cdots +r_t}))
\]
for an $A_n$-space $(X,\{m_i\})$.

Thus, if $f:X\to X'$ is an $A_n$-homomorphism between $A_n$-spaces $(X,\{m_i\})$ and $(X',\{m_i'\})$, then the sequence of maps $\{f_i:\mathcal{J}_i\times X^i\to X'\}_{i=1}^n$ defined by
\[
f_i(\rho ;x_1,\ldots ,x_i)=f(m_i(\pi _i(\rho );x_1,\ldots ,x_i))
\]
is an $A_n$-form of $f$.
We call $\{f_i\}$ the \textit{canonical $A_n$-form} of $f$.
Unless otherwise stated, we regard an $A_n$-homomorphism as an $A_n$-map with the canonical $A_n$-form.
\begin{ex}
A homomorphism $f:G\to G'$ between topological monoids is an $A_\infty$-homomorphism.
The canonical $A_\infty$-form $\{f_i\}$ of $f$ is given as
\[
f_i(\rho ;x_1,\ldots ,x_i)=f(x_1\cdots x_i).
\]
\end{ex}
\par
Next, we consider the composition of $A_n$-maps.
The composition of general $A_n$-maps is complicated.
But the composite of an $A_n$-map with an $A_n$-homomorphism is easily considered as follows.
\begin{dfn}
Let $(X,\{m_i\}),(X',\{m_i'\}),(X'',\{m_i''\})$ be $A_n$-spaces.
If $(f,\{f_i\}):(X,\{m_i\})\to (X',\{m_i'\})$ is an $A_n$-map and $g:X'\to X''$ is an $A_n$-homomorphism, then we call the $A_n$-map $(g\circ f,\{g\circ f_i\})$ the \textit{canonical composite} of $(f,\{f_i\})$ and $g$.
Similarly, if $f:X\to X'$ is an $A_n$-homomorphism and $(g,\{g_i\}):(X',\{m_i'\})\to (X'',\{m_i''\})$ is an $A_n$-map, then we call the $A_n$-map $(g\circ f,\{g_i\circ ({\rm id}_{\mathcal{J}_i}\times f^{\times i})\})$ the \textit{canonical composite} of $f$ and $(g,\{g_i\})$.
\end{dfn}
We also consider an equivalence relation of $A_n$-spaces.
\begin{dfn}
Let $(X,\{m_i\}),(X',\{m_i'\})$ be $A_n$-spaces.
Then an $A_n$-map $(f,\{f_i\}):(X,\{m_i\})\to (X',\{m_i'\})$ is said to be an \textit{$A_n$-equivalence} if the underlying map $f$ is a pointed homotopy equivalence.
We say $(X,\{m_i\})$ and $(X',\{m_i'\})$ are \textit{$A_n$-equivalent} if there exists an $A_n$-equivalence $(X,\{m_i\})\to (X',\{m, _i'\})$.
\end{dfn}
\begin{prp}\label{inverseofanequiv}
The relation $A_n$-equivalence is an equivalence relation between $A_n$-spaces.
\end{prp}
\begin{rem}
Iwase and Mimura gave an outline of the proof in \cite{IM89}.
We will give a proof of this proposition in Remark \ref{inverseofanequiv_rem}.
\end{rem}

For a pointed space $X$, let us denote the subspace of the product $X^n$ consisting of the points $(x_1,\ldots ,x_n)\in X^n$ with $x_k=*$ for some $k$ by $X^{[n]}$ (the \textit{fat wedge}).
The localization of an $A_n$-space is also an $A_n$-space as follows.
\begin{prp}\label{localizationofA_n}
Let $\mathcal{P}$ be a set of primes and $(X,\{m_i\})$ be a path-connected $A_n$-space.
For a $\mathcal{P}$-localization map $\ell :X\to X_{\mathcal{P}}$, there exists an $A_n$-form $\{m_i^{\mathcal{P}}\}$ on $X_{\mathcal{P}}$ such that $\ell$ admits an $A_n$-form $\{\ell _i\}$ such that $(\ell,\{\ell_i\})$ is an $A_n$-map between $(X,\{m_i\})$ and $(X_{\mathcal{P}},\{m_i^{\mathcal{P}}\})$.
\end{prp}
\begin{proof}
We construct the $A_n$-forms $\{m_i^{\mathcal{P}}\}$ and $\{\ell _i\}$ inductively.
Put $\ell _1=\ell$.
Assume we have obtained $A_{i-1}$-forms $\{m_j^{\mathcal{P}}\}_{j=2}^{i-1}$ and $\{\ell _j\}_{j=1}^{i-1}$ as above.
We define a map $\ell _i$ on $(\partial \mathcal{J}_i-\mathrm{Int}\delta (i,1,\ldots ,1))\times X^i\cup \mathcal{J}_i\times X^{[n]}$ by
\begin{align*}
	&\ell _i(\rho ;x_1,\ldots ,x_i) \\
	=&\left\{
	\begin{array}{ll}
		\ell _s(\sigma ;x_1,\ldots ,x_{k-1},m_t(\tau ;x_k,\ldots ,x_{k+s-1}),x_{k+s},\ldots ,x_i) & (\rho =\delta _k(s,t)(\sigma ,\tau )) \\
		m_t^{\mathcal{P}}(\tau ;\ell _{s_1}(\sigma _1;x_1,\ldots ,x_{s_1}),\ldots ,\ell _{s_t}(\sigma _t;x_{s_1+\cdots +s_{i-1}+1},\ldots ,x_i)) & (\rho =\delta (t,s_1,\ldots ,s_t)(\tau ,\sigma _1,\ldots ,\sigma _t),t<i) \\
		\ell _{i-1}(d_k(\rho );x_1,\ldots ,x_{k-1},x_{k+1},\ldots ,x_i) & (x_k=*)
	\end{array}
	\right. .
\end{align*}
Since $X^{[i]}\subset X^i$ has homotopy extension property and $\partial \mathcal{J}_i-\mathrm{Int}\delta (i,1,\ldots ,1)$ is a deformation retract of $\mathcal{J}_i$, this map extends to $\ell _i:\mathcal{J}_i\times X^i\to X_{\mathcal{P}}$.
Similarly, we define $m_i^{\mathcal{P}}$ on $\partial \mathcal{K}_i\times X^i_{\mathcal{P}}\cup \mathcal{K}_i\times X^{[n]}_{\mathcal{P}}$ by
\[
	m_i^{\mathcal{P}}(\rho ;x_1,\ldots ,x_i)=\left\{
	\begin{array}{ll}
	m_s^{\mathcal{P}}(\sigma ;x_1,\ldots ,x_{k-1},m_t(\tau ;x_k,\ldots ,x_{k+s-1}),x_{k+s},\ldots ,x_i) & (\rho =\partial _k(s,t)(\sigma ,\tau )) \\
	m_{i-1}^{\mathcal{P}}(s_k(\rho );x_1,\ldots ,x_{k-1},x_{k+1},\ldots ,x_i) & (x_k=*).
	\end{array}
	\right.
\]
By the weak homotopy equivalence $(\ell ^i)^*:\Map (X^i_{\mathcal{P}},X_{\mathcal{P}})\to \Map (X^i,X_{\mathcal{P}})$, $m_i^{\mathcal{P}}$ extends to $m_i^{\mathcal{P}}:\mathcal{K}_i\times X^i_{\mathcal{P}}\to X_{\mathcal{P}}$ such that
\[
	\ell _i(\delta (i,1,\ldots ,1)(\rho ,*,\ldots ,*);x_1,\ldots ,x_i)=m_i^{\mathcal{P}}(\rho ;\ell (x_1),\ldots ,\ell (x_i)).
\]
Therefore, $(\ell,\{\ell_j\}_{j=1}^i)$ is an $A_i$-map between $(X,\{m_j\}_{j=2}^i)$ and $(X_{\mathcal{P}},\{m_j^{\mathcal{P}}\}_{j=2}^i)$.
\end{proof}
\begin{rem}
In the above proof, we only used the property of the $\mathcal{P}$-localization that $\ell:X\to X_{\mathcal{P}}$ induces the weak homotopy equivalence $\Map(X^i_{\mathcal{P}},Y)\to \Map(X^i,Y)$ for any $\mathcal{P}$-local space $Y$.
\end{rem}
\begin{rem}
Every path-connected $A_n$-space $(n\geq 2)$ is nilpotent and so it has its $\mathcal{P}$-localization in the sense of \cite{HMR75}.
\end{rem}

\section{Pullback of $A_n$-homomorphism fibrations}\label{pullback1}
From now on, we often abbreviate $A_n$-forms of $A_n$-spaces and $A_n$-maps by underlying spaces and maps, respectively.
For example, an $A_n$-space $(X,\{m_i\})$ is abbreviated by $X$, an $A_n$-map $(f,\{f_i\}):(X,\{m_i\})\to (X',\{m_i'\})$ by $f:X\to X'$, a homotopy $(F,\{F_i\})$ between $A_n$-maps $(f,\{f_i\}),(f',\{f_i'\}):(X,\{m_i\})\to (X',\{m_i'\})$ by $F:I\times X\to X'$, and so on.

As easily checked, for $A_n$-homomorphisms $f_1:X_1\to X_3$ and $f_2:X_2\to X_3$, the pullback (in the topological sense) $X$ of the diagram $X_1\xrightarrow{f_1}X_3\xleftarrow{f_2}X_2$ naturally inherits an $A_n$-form and the natural projections $p_1:X\to X_1$ and $p_2:X\to X_2$ are $A_n$-homomorphisms.
If the map $f_1:X_1\to X_3$ is a Hurewicz fibration, $X$ has the following universal property.
\begin{thm}\label{homomorphismpullback}
Let $X_1$, $X_2$ and $X_3$ be $A_n$-spaces and $f_1:X_1\to X_3$ and $f_2:X_2\to X_3$ be $A_n$-homomorphisms.
In addition, suppose $f_1$ is a Hurewicz fibration.
Denote the topological pullback of $X_1\xrightarrow{f_1}X_3\xleftarrow{f_2}X_2$ by $X$ and the projections by $p_1:X\to X_1$ and $p_2:X\to X_2$.
Then $X$ has the following homotopy lifting property:
for an $A_n$-space $Y$ and $A_n$-maps $g_1:Y\to X_1$ and $g_2:Y\to X_2$, if the canonical composites $f_1\circ g_1$ and $f_2 \circ g_2$ are homotopic as $A_n$-maps through a homotopy $G_3:I\times Y\to X_3$, then there exist an $A_n$-map $g:Y\to X$ and a homotopy $G_1:I\times Y\to X_1$ of $A_n$-maps from $g_1$ to $p_1\circ g$, where the canonical composites $p_2\circ g$ and $f_1\circ G_1$ are equal to $g_2$ and $G_3$, respectively.
Here, $g$ and $G_1$ are unique up to homotopy.
More precisely, if $g':Y\to X$ and $G_1':I\times Y\to X_1$ satisfy the same condition as $g$ and $G_1$, then there exist a homotopy $\gamma :I\times Y\to X$ of $A_n$-maps from $g$ to $g'$ and a homotopy $\varGamma _1:I\times (I\times Y)\to X_1$ of homotopies of $A_n$-maps from $G_1$ to $G_1'$ such that the canonical composites $p_2\circ \gamma$ and $f_1\circ \varGamma _1$ are the stationary homotopies of $g_2$ and $G_3$, respectively, and the canonical composite $p_1\circ \gamma$ is equal to $\varGamma |_{I\times (\{1\}\times Y)}$.
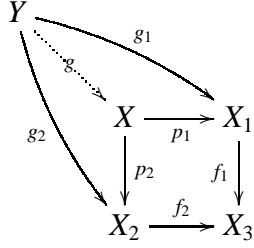
\[
\xymatrix{
Y \ar@/_/[ddr]_{g_2} \ar@/^/[drr]^{g_1}
\ar@{.>}[dr]|-{g} \\
& X \ar[d]^{p_2} \ar[r]_{p_1}
& X_1 \ar[d]_{f_1} \\
& X_2 \ar[r]^{f_2} & X_3
}
\]
\end{thm}
This theorem immediately follows from the property of the topological pullback and the following lemma.
\begin{lem}
Let $X_1$, $X_3$ and $Y$ be $A_n$-spaces and $f_1:X_1\to X_3$ be an $A_n$-homomorphism and a Hurewicz fibration.
If $g_1:Y\to X_1$ and $g_3:Y\to X_3$ are $A_n$-maps such that there is a homotopy $G_3:I\times Y\to X_3$ of $A_n$-maps from $f_1\circ g_1$ to $g_3$, then there exists a homotopy $G_1:I\times Y\to X_1$ of $A_n$-maps such that the canonical composite $f_1\circ G_1$ is equal to $G_3$.
\par
Moreover, if $G_1':I\times Y\to X_1$ satisfies the same condition as $G_1$, then there exists a homotopy $\varGamma _1:I\times (I\times Y)\to X$ of homotopies of $A_n$-maps from $G_1$ to $G_1'$ such that the restriction $\varGamma_1|_{I\times(\{0\}\times Y)}$ is the stationary homotopy of $g_1$ and the canonical composite $f_1\circ \varGamma _1$ is the stationary homotopy of $G_3$.
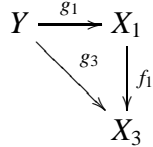
\[
	\xymatrix{
		Y \ar[r]^-{g_1} \ar[dr]^-{g_3} & X_1 \ar[d]^-{f_1} \\
		& X_3 \\
	}
\]
\end{lem}
\begin{proof}
Denote the $A_n$-forms of $X_1$, $Y$, $g_1$, $G_3$ by $\{(m_1)_i\}$, $\{(m_Y)_i\}$ $\{(g_1)_i\}$ and $\{(G_3)_i\}$, respectively.
We prove this lemma by induction.
As a homotopy of pointed maps, $G_3$ lifts to $G_1$ by the covering homotopy property of $f_1$.
Assume $G_3$ lifts to a homotopy of $A_{i-1}$-maps $\{(G_1)_j\}_{j=1}^{i-1}$.
Then define a map $(G_1)_i:(((\{0\}\times \mathcal{J}_i)\cup (I\times \partial \mathcal{J}_i))\times Y^i)\cup (I\times \mathcal{J}_i\times Y^{[i]})\to X_1$ by
\begin{align*}
&(G_1)_i(u,\rho ;y_1,\ldots ,y_i) \\
=&\left\{
\begin{array}{ll}
(g_1)_i(\rho ;y_1,\ldots ,y_i) & (u=0) \\
(G_1)_s(\sigma ;y_1,\ldots ,y_{k-1},(m_Y)_t(\tau ;y_k,\ldots ,y_{k+s-1}),y_{k+s},\ldots ,y_i) & (\rho =\delta _k(s,t)(\sigma, \tau )) \\
(m_1)_s(\sigma ;(G_1)_{t_1}(\tau _1;y_1,\ldots ,y_{t_1}),\ldots ,(G_1)_{t_s}(\tau _s;y_{t_1+\cdots +t_{s-1}+1},\ldots ,y_i)) & (\rho =\delta (s,t_1,\ldots ,t_s)(\sigma ,\tau _1,\ldots ,\tau _s)) \\
(G_1)_{i-1}(u,d_k(\rho );y_1,\ldots ,y_{k-1},y_{k+1},\ldots ,y_i) & (y_k=*) \\
\end{array}
\right. .
\end{align*}
Then the composite $f_1\circ (G_1)_i$ is equal to $(G_3)_i$ at every point where $(G_1)_i$ is defined.
By the covering homotopy extension property, $(G_1)_i$ extends to $(G_1)_i:I\times \mathcal{J}_i\times Y^i\to X_1$ with $f_1\circ (G_1)_i=(G_3)_i$.
Therefore, $G_3$ lifts to a homotopy of $A_{i}$-maps $\{(G_1)_j\}_{j=1}^{i}$ satisfying the desired properties.

The latter half follows from a quite analogous argument.
\end{proof}
One can prove the universal property for homotopies as follows analogously.
For homotopies $H,H':I\times Y\to X$ of $A_n$-maps between $f$ and $f'$ and between $f'$ and $f''$, respectively, let us denote the composite of homotopies $H$ and $H'$ by $H*H'$.
More precisely, the underlying homotopy $H*H':I\times (I\times Y)\to X$ is defined as
\[
	(H*H')(v,u;y)=\left\{
	\begin{array}{ll}
		H(2v,u;y) & (0\leq v\leq 1/2) \\
		H'(2v-1,u;y) & (1/2\leq v\leq 1) \\
	\end{array}
	\right.
\]
and its $A_n$-form is similarly defined.
\begin{thm}\label{homomorphismpullback2}
Set the following diagram of pullback of $A_n$-homomorphisms as in Theorem \ref{homomorphismpullback}.
\[
	\xymatrix{
		X \ar[d]^{p_1} \ar[r]_{p_2} & X_2 \ar[d]_{f_2} \\
		X_1 \ar[r]^{f_1} & X_3 \\
	}
\]
For an $A_n$-space $Y$, $A_n$-maps $g_1,g_1':Y\to X_1$, $g_2,g_2':Y\to X_2$, and homotopies of $A_n$-maps $h_1:I\times Y\to X_1$ from $g_1$ to $g_1'$, $h_2:I\times Y\to X_2$ from $g_2$ to $g_2'$, $G_3:I\times Y\to X_3$ from $f_1\circ g_1$ to $f_2\circ g_2$ and $G_3':I\times Y\to X_3$ from $f_1\circ g_1'$ to $f_2\circ g_2'$, if the composite of homotopies $(f_1\circ h_1)*G_3'$ and $G_3*(f_2\circ h_2)$ are homotopic as homotopies of $A_n$-maps, then there exist a homotopy of $A_n$-maps $\gamma :I\times Y\to X$ and a homotopy of homotopies of $A_n$-maps $\varGamma _1:I\times (I\times Y)\to X_1$ such that the following equalities as homotopies of $A_n$-maps hold:
\begin{align*}
	\varGamma _1|_{\{ 0\} \times (I\times Y)}&=h_1, \\
	\varGamma _1|_{\{ 1\} \times (I\times Y)}&=p_1\circ \gamma , \\
	f_1\circ \varGamma _1|_{I\times (\{ 0\} \times Y)}&=G_3, \\
	f_1\circ \varGamma _1|_{I\times (\{ 1\} \times Y)}&=G_3'.
\end{align*}
\end{thm}
Homotopies in this theorem are described as in the following diagram which commutes up to homotopy between homotopies of $A_n$-maps.
\begin{align*}
	\xymatrix{
		f_1\circ g_1 \ar@{=>}[r]^-{G_3} \ar@{=>}[d]_-{f_1\circ h_1} & f_2\circ g_2 \ar@{=>}[d]^-{f_2\circ h_2} \\
		f_1\circ g_1' \ar@{=>}[r]^-{G_3'} & f_2\circ g_2'
	}
\end{align*}
The arrows of the shape $\Rightarrow$ represent homotopies.

\section{Homotopy pullback of general $A_n$-maps}\label{pullback2}
In the previous section, we observed the pullback of $A_n$-homomorphisms, one of which is a Hurewicz fibration.
To generalize this construction, we show the fact that every $A_n$-map can be ``replaced'' by an $A_n$-homomorphism which is a Hurewicz fibration.
\par
Let $f:X\to X'$ be a pointed map between pointed spaces.
As is well-known, there is a Hurewicz fibration $\tilde f:\tilde X\to X'$ and a pointed homotopy equivalence $q:\tilde X\to X$ such that the following diagram commutes up to homotopy,
\[
	\xymatrix{
		\tilde X \ar[dr]^-{\tilde f} \ar[d]_-q & \\
		X \ar[r]_-{f} & X' \\
	}
\]
where
\begin{align*}
&\tilde X=\{\, (x,\ell )\in X\times X'^I\,|\, \ell (0)=f(x) \,\},\\
&q(x,\ell )=x,\hspace{1em}\tilde f(x,\ell )=\ell (1).
\end{align*}
We remark that $q$ is also a Hurewicz fibration.
\begin{prp}\label{replacement}
Let $(X,\{m_i\})$ and $(X',\{m_i'\})$ be $A_n$-spaces, $(f,\{f_i\}):(X,\{m_i\})\to (X',\{m_i'\})$ an $A_n$-map and $\tilde X$, $\tilde f$, $q$ as above.
Then $\tilde X$ admits an $A_n$-form such that $\tilde f$ and $q$ are $A_n$-homomorphisms and the canonical composite $(f\circ q,\{f_i\circ(\mathrm{id}_{\mathcal{J}_i}\times q^{\times i})\})$ is homotopic to $\tilde f$ as an $A_n$-map.
\end{prp}
\begin{proof}
We construct maps $F_i:I\times \mathcal{J}_i\times \tilde X^i\to X'$ and $\tilde m_i:\mathcal{K}_i\times \tilde X^i\to \tilde X$ inductively as follows.
Define a map $F_1:I\times \tilde X\to X'$ by $F_1(u,(x,\ell))=\ell(u)$.
Suppose that we have defined the maps $\{F_j\}_{j=1}^{i-1}$ and $\{\tilde m_j\}_{j=2}^{i-1}$ such that
\begin{align*}
&F_j(u,\rho ;(x_1,\ell_1 ),\ldots ,(x_j,\ell _j))\\
=&\left\{
\begin{array}{ll}
f_j(\rho;x_1,\ldots,x_j) & (u=0) \\
M_j'(\rho;\ell_1(1),\ldots ,\ell_j(1)) & (u=1) \\
F_s(u,\sigma ;(x_1,\ell _1),\ldots ,\tilde m_t(\tau ;(x_k,\ell _k),\ldots ),\ldots ,(x_j,\ell _j)) & (\rho =\delta _k(s,t)(\sigma ,\tau ),s>1) \\
m_s'(\sigma ;F_{t_1}(u,\tau _1;(x_1,\ell _1),\ldots ),\ldots ,F_{t_s}(u,\tau _s;\ldots ,(x_j,\ell _j))) & (\rho =\delta (s,t_1,\ldots ,t_s)(\sigma ,\tau _1,\ldots ,\tau_s))\\
F_{j-1}(u,s_k(\rho );(x_1,\ell _1),\ldots ,(x_{k-1},\ell _{k-1}),(x_{k+1},\ell _{k+1}),\ldots ,(x_j,\ell _j)) & ((x_k,\ell _k)=*)\\
\end{array}
\right.
\end{align*}
and
\[
\tilde m_j(\rho ;(x_1,\ell _1),\ldots ,(x_j,\ell _j))=(m_j(\rho ;x_1,\ldots ,x_j),\, (u\mapsto F_j(u,\delta _1(1,j)(*,\rho );(x_1,\ell_1),\ldots ,(x_j,\ell _j)))\,),
\]
where $\{M_j'\}$ is the canonical $A_n$-form of the identity map $X'\to X'$.
Let us denote $\partial (I\times \mathcal{J}_i)=(\partial I\times \mathcal{J}_i)\cup (I\times \partial \mathcal{J}_i)$.
We also define $F_i$ on $(\partial(I\times \mathcal{J}_i)-{\rm Int}\,(I\times \delta_1(1,i)))\times \tilde X^i\cup I\times \mathcal{J}_i\times \tilde X^{[i]}$ by the above formula.
Since $\partial(I\times \mathcal{J}_i)-{\rm Int}\,(I\times \delta_1(1,i))$ is a deformation retract of $I\times \mathcal{J}_i$ and $\tilde X$ is well-pointed, $F_i$ extends over $I\times \mathcal{J}_i\times \tilde X^i$ and then $\tilde m_i$ is obtained by the above formula.
\par
Thus $\{\tilde m_i\}_{i=2}^n$ is an $A_n$-form of $\tilde X$ and $\tilde f$ is an $A_n$-homomorphism.
The family of maps $\{F_i\}$ is a homotopy from $\{f_i\circ ({\rm id}_{\mathcal{J}_i}\times q^{\times i})\}$ to $\{M_i'\circ ({\rm id}_{\mathcal{J}_i}\times \tilde f^{\times i})\}$ through $A_n$-forms, where $\{M_i'\circ ({\rm id}_{\mathcal{J}_i}\times \tilde f^{\times i})\}$ is the canonical $A_n$-form of $\tilde f$.
\end{proof}
\begin{rem}\label{replace}
	If $X$ and $X'$ are topological monoids and $f$ is a homomorphism, then $\tilde X$ is naturally a topological monoid and $\tilde f$ is homotopic to $f\circ q$ as a homomorphism.
\end{rem}

\begin{rem}\label{inverseofanequiv_rem}
In fact, using Proposition \ref{replacement}, we can show Proposition \ref{inverseofanequiv} as follows.
Let $f:X\to X'$ be an $A_n$-equivalence between $A_n$-spaces.
Take $\tilde f:\tilde X\to X'$ and $q:\tilde X\to X$ as above.
Since $\tilde f:\tilde X\to X'$ is a homotopy equivalence fibration and an $A_n$-homomorphism, we can find an $A_n$-map $s:X'\to \tilde X$ such that the canonical composite $\tilde f\circ s$ is the identity on $X'$ with the canonical $A_n$-form.
A homotopy inverse $q\circ s$ of $f$ thus admits an $A_n$-form.
\end{rem}

Now we recall homotopy pullback.
For a diagram $X_1\xrightarrow{f_1}X_3\xleftarrow{f_2}X_2$, the \textit{homotopy pullback} $X$ of this diagram is the topological pullback of the diagram $\tilde X_1\xrightarrow{\tilde f_1}X_3\xleftarrow{\tilde f_2}\tilde X_2$.
There are natural projections $q_i:X\to X_i$ $(i=1,2)$, which is defined as the composite of the canonical projections $p_i:X\to \tilde X_i$ and $\tilde q_i:\tilde X_i\to X_i$.
If there exists the following homotopy commutative diagram:
\[
	\xymatrix{
		X_1 \ar[r] \ar[d] & X_3 \ar[d] & \ar[l] \ar[d] X_2 \\
		Y_1 \ar[r] & Y_3 & Y_2 \ar[l]
	}
\]
then there exists a lift $X\to Y$.
Moreover, if all the vertical arrows are homotopy equivalence, then the lift $X\to Y$ is also a homotopy equivalence.

By Proposition \ref{replacement}, we immediately obtain the following theorem.

\begin{thm}\label{lifting1}
Let $f_1:X_1\to X_3$ and $f_2:X_2\to X_3$ be $A_n$-maps between $A_n$-spaces.
Then the homotopy pullback $X$ of $X_1\xrightarrow{f_1}X_3\xleftarrow{f_2}X_2$ admits an $A_n$-form such that the natural projections $q_1:X\to X_1$ and $q_2:X\to X_2$ are $A_n$-homomorphisms and the canonical composites $f_1\circ q_1:X\to X_3$ and $f_2\circ q_2:X\to X_3$ are homotopic as $A_n$-maps.
\end{thm}

\begin{rem}
We have shown the universal property of homotopy pullback of $A_n$-homomorphisms in Theorem \ref{homomorphismpullback}.
In contrast, Theorem \ref{lifting1} does not claim any universal property of homotopy pullback of general $A_n$-maps.
To state the universal property precisely, we need the ``higher category of $A_n$-spaces and $A_n$-maps with higher homotopy''.
For example, Lurie \cite{Lur} formulated operads in a higher categorical language and showed some commutativity of general homotopy limits with forgetting operad actions.
But to apply his result to our situation, it seems to take too much effort to translate it to our language.
So we choose to work in a very naive formulation.
\end{rem}

\begin{ex}
Let $f:X\to Y$ be an $A_n$-map.
The \textit{homotopy fiber} $F$ of $f$ is defined by the following homotopy pullback square
\begin{align*}
	\xymatrix{
		F \ar[r] \ar[d] & X \ar[d]^-{f} \\
		\ast \ar[r] & Y.
	}
\end{align*}
Then, since the point $*$ and the inclusion $*\to Y$ have the unique $A_n$-forms, $F$ and the inclusion $F\to X$ are naturally an $A_n$-space and an $A_n$-map, respectively.
\end{ex}

As an application, we revisit the result of Zabrodsky \cite{Zab70}.
\begin{thm}[\cite{Zab70}]
For any prime $p$, there is a finite CW complex $X$ which admits an $A_{p-1}$-form but no $A_p$-form.
\end{thm}
\begin{proof}
For the case when $p=2$, it is sufficient to take $X=S^2$.
Let $p$ be an odd prime.
Recall the fact that the double suspension map $\Sigma ^2:S^{2n-1}\to \Omega ^2S^{2n+1}$ induces the isomorphisms $\pi _i(S^{2n-1})_{(p)}\to \pi _i(\Omega ^2S^{2n+1})_{(p)}$ on $p$-localized homotopy groups for $i<2pn-3$.
This implies that the $p$-localized sphere $S^{2n-1}_{(p)}$ admits an $A_{p-1}$-form since $\Omega ^2S^{2n+1}_{(p)}$ is an $A_\infty$-space.

Next, consider the homotopy pullback $Y$ of the diagram $S^3_{(p)}\times \cdots \times S^{2p+1}_{(p)}\to K(\mathbb{Q},3)\times \cdots \times K(\mathbb{Q},2p+1)\leftarrow \SU(p+1)_{\mathcal{P}}$ of localizations of $A_{p-1}$-spaces, where $\mathcal{P}$ is the set of the primes other than $p$.
This is justified since $\SU(p+1)_{(0)}$ is $A_\infty$-equivalent to $K(\mathbb{Q},3)\times \cdots \times K(\mathbb{Q},2p+1)$ (because the classifying space $B\SU(p+1)_{(0)}$ is homotopy equivalent to $K(\mathbb{Q},4)\times \cdots \times K(\mathbb{Q},2p+2)$) and for each $i$, the rationalization $S^{2i-1}_{(p)}\to K(\mathbb{Q},2i-1)$ admits an $A_{p-1}$-form by an easy obstruction argument.
As easily seen, $Y$ is homotopy equivalent to a finite CW complex of which cells have the same dimensions as ones of $S^3\times \cdots \times S^{2p-1}$ and $\SU(p+1)$.
By Theorem \ref{lifting1}, $Y$ admits an $A_{p-1}$-form.

Consider the $p$-localization $Y_{(p)}\simeq S^3_{(p)}\times \cdots \times S^{2p+1}_{(p)}$.
Suppose $Y$ admits an $A_p$-form.
Then $Y_{(p)}$ also admits an $A_p$-form (Proposition \ref{localizationofA_n}).
This contradicts to the triviality of the action of the Steenrod operations on $H^*(Y_{(p)};\mathbb{F}_p)$ by Lemma \ref{steenrod_on_A_p} below.
Thus $Y$ admits no $A_p$-form.
\end{proof}

\begin{lem}\label{steenrod_on_A_p}
Let $X$ be an $A_p$-space with $p$ an odd prime.
If the mod $p$ cohomology of $X$ is the exterior algebra
\begin{align*}
H^*(X;\mathbb{F}_p)=\wedge_{\mathbb{F}_p}(x_3,x_5,\ldots,x_{2p+1})
\end{align*}
with $x_{2i+1}\in H^{2i+1}(X;\mathbb{F}_p)$, then there exists a unit $c\in \mathbb{F}_p-\{0\}$ such that $\mathcal{P}^1x_3-cx_{2p+1}$ is decomposable.
\end{lem}
\begin{proof}
We assume the results of \cite{Iwa84} and \cite{Hem91}.
Let us denote the $j$-th projective space of $X$ by $XP^j$.
Then we have a decomposition of an algebra
\begin{align*}
	H^*(XP^{p-1};\mathbb{F}_p)=\mathbb{F}_p[z_4,z_6,\ldots,x_{2p+2}]/(z_{2i_1+2}\cdots z_{2i_p+2}\mid 1\leq i_1\leq\cdots\leq i_p\leq p)\oplus S,
\end{align*}
where $z_{2i+2}$ is the transgression of $x_{2i+1}$ with respect to the canonical homotopy fibration $X\to \Sigma^{p-1}X^{\wedge p}\to XP^{p-1}$ and $S$ is some ideal.
Moreover, every non-trivial element in $S$ has degree $\geq 4p+1$.
Consider the canonical homotopy cofibration
\begin{align*}
	\Sigma^{p-1}X^{\wedge p}\to XP^{p-1}\to XP^p,
\end{align*}
where the left map is compatible with the projection in the homotopy fibration mentioned above.
Since the non-zero elements in $\tilde H^*(\Sigma^{p-1}X^{\wedge p};\mathbb{F}_p)$ have degree $\geq 4p-1$, the induced map $H^*(XP^p;\mathbb{F}_p)\to H^*(XP^{p-1};\mathbb{F}_p)$ is an isomorphism for $*\leq 4p-2$.
Then denote the corresponding element to $z_{2i+2}$ by the same symbol $z_{2i+2}\in H^{2i+2}(XP^p;\mathbb{F}_p)$.
The cohomology $H^{4p}(XP^p;\mathbb{F}_p)$ is spanned by the elements in $\mathbb{F}_p[z_4,z_6,\ldots,x_{2p+2}]$ of degree $4p$ and $z_4^p$ is non-trivial because $z_4^p$ is contained in the image of the connecting map $H^{4p-1}(\Sigma^{p-1}X^{\wedge p};\mathbb{F}_p)\to H^{4p}(XP^p;\mathbb{F}_p)$ and every non-trivial element in $S$ has degree $\geq 4p+1$.

Consider the element $\mathcal{P}^1z_4\in H^{2p+2}(XP^p;\mathbb{F}_p)$.
By the Adem relation, we have
\begin{align*}
	z_4^p=\mathcal{P}^2z_4=2\mathcal{P}^1\mathcal{P}^1z_4
\end{align*}
and thus $\mathcal{P}^1z_4\not=0$.
The element $\mathcal{P}^1z_4$ is written as a non-trivial linear combination of monomials in $\mathbb{F}_p[z_4,z_6,\ldots,x_{2p+2}]$ because of the degrees of non-trivial elements in $S$ and $\tilde H^*(\Sigma^{p-1}X^{\wedge p};\mathbb{F}_p)$.
Then there is a monomial $z_{2i_1+2}\cdots z_{2i_k+2}$ such that $z_4^p$ appears in the expression of $\mathcal{P}^1(z_{2i_1+2}\cdots z_{2i_k+2})$ modulo the ideal $(z_6,\ldots,z_{2p+2})$.
But by the Cartan formula, we have $k=1$ and thus obtain $i_1=2p+2$ by degree.
Therefore, the element $\mathcal{P}^1z_4-cz_{2p+2}$ is decomposable for some $c\in\mathbb{F}_p-\{0\}$ and we conclude the assertion of this proposition.
\end{proof}

\section{Framed fiberwise homotopy theory}\label{framedfiberwisehomotopytheory}
From this section to Section \ref{classificationsection}, we devote ourselves to preparing the tools to handle fiberwise $A_n$-spaces for our later applications to gauge groups.
\par
First of all, we recall the terminology of fiberwise homotopy theory \cite{CJ98}.
\begin{dfn}
Let $B$ be a space.
\begin{enumerate}
\item
	A \textit{fiberwise space} over $B$ consists of a space $E$ together with a map $\pi :E\to B$, called the \textit{projection}.
\item
	For fiberwise spaces $E\stackrel{\pi}{\to}B$ and $E'\stackrel{\pi'}{\to}B$ and a map $f:E\to E'$, $f$ is called a \textit{fiberwise map} over $B$ if $\pi '\circ f=\pi$.
\item
	For fiberwise spaces $E\stackrel{\pi}{\to}B$ and $E'\stackrel{\pi'}{\to}B$ and fiberwise maps $f_0,f_1:E\to E'$, a homotopy $f:I\times E\to E'$ between $f_0$ and $f_1$ is called a \textit{fiberwise homotopy} if $\pi'\circ f=\pi \circ p_2$, where $p_2:I\times E\to E$ is the second projection.
\item
	A \textit{fiberwise pointed space} over $B$ consists of a fiberwise space $E\stackrel{\pi}{\to}B$ together with a section $\sigma :B\to E$ of $\pi$ (i.e. $\pi \circ \sigma ={\rm id}_B$).
	Moreover, we require fiberwise pointed spaces \textit{fiberwise well-pointedness}:
	for a fiberwise map $f:E\to E'$, every fiberwise homotopy $h:I\times \sigma (B)\to E'$ with $h|_{\{0\}\times \sigma (B)}=f$ extends to a fiberwise homotopy $h':I\times E\to E'$ such that $h'|_{\{0\}\times E}=f$.
	Note that each fiber $E_b$ of a fiberwise pointed space $E$ is considered as a pointed space with basepoint $\sigma (b)$.
\item
	For fiberwise pointed spaces $B\stackrel{\sigma}{\to}E\stackrel{\pi}{\to}B$ and $B\stackrel{\sigma'}{\to}E'\stackrel{\pi'}{\to}B$ and a map $f:E\to E'$, $f$ is called a \textit{fiberwise pointed map} over $B$ if $\pi '\circ f=\pi$ and $f\circ \sigma =\sigma '$.
\item
	\textit{Fiberwise pointed homotopies} are defined to be section-preserving fiberwise homotopies in the obvious way.
\item
	For fiberwise spaces $E\stackrel{\pi}{\to}B$ and $E'\stackrel{\pi'}{\to}B$, the \textit{fiber product} $E\times_B E'$ is defined as
	\[
		E\times _BE'=\{\, (e,e')\in E\times E'\,|\, \pi (e)=\pi '(e')\,\}
	\]
	with the obvious projection $E\times _BE'\to B$.
	We denote the $n$-fold fiber product of $E$ by $E^{\times _Bn}$.
	The fiber product of fiberwise pointed spaces is defined as a fiber product of them with obvious section.
\item
	A fiberwise pointed space $B\stackrel{\sigma}{\to}E\stackrel{\pi}{\to}B$ is said to be an \textit{ex-fibration} if it has the following \textit{pointed homotopy lifting property}:
	for any base space $B'$ and fiberwise pointed space $B'\xrightarrow{\sigma '}E'\xrightarrow{\pi '}B'$ over $B'$, if $F:I\times B'\to B$ is a homotopy and a map $f_0:E'\to E$ satisfies $\pi \circ f_0=(F|_{\{0\}\times B'})\circ \pi'$ and $f_0\circ \sigma '=\sigma \circ (F|_{\{0\}\times B'})$, there is a homotopy $f:I\times E'\to E$ covering $F$ (i.e. $\pi \circ f=F\circ ({\rm id}_I\times \pi ')$) such that $f|_{\{0\}\times E'}=f_0$ and $f\circ (\mathrm{id}_I\times \sigma )=\sigma \circ F$.
\end{enumerate}
\end{dfn}
\begin{rem}
In the rest of this section, we quote May's result in \cite{May75}.
The terminology he used are different from the above.
Let $\mathcal{U}$ be the category of compactly generated spaces and $\mathcal{T}$ be the category of well-pointed compactly generated spaces.
We list corresponding his terminology here:
\begin{enumerate}
\item
	$\mathcal{U}$-space,
\item
	$\mathcal{U}$-map,
\item
	$\mathcal{U}$-homotopy,
\item
	$\mathcal{T}$-space,
\item
	$\mathcal{T}$-map,
\item
	$\mathcal{T}$-homotopy,
\item
	(he did not used fiber product in \cite{May75}),
\item
	$\mathcal{T}$-fibration.
\end{enumerate}
We remark that May's $\mathcal{U}$-fibration means just Hurewicz fibration.
\end{rem}
Let us introduce framed fibrations.
Since the unframed theory has already been established, we concentrate on the case such that the base space is path-connected.
\begin{dfn}
Fix a space $X$, a Hurewicz fibration $E\to B$ over a path-connected pointed space $B$ is said to be \textit{$X$-framed} if a homotopy equivalence, which we call \textit{framing}, from $X$ to the fiber $E_*$ over the basepoint is given.
For $X$-framed Hurewicz fibrations $E$ and $E'$ over $B$, a fiberwise map $f:E\to E'$ is said to be \textit{$X$-framed} if the following diagram commutes up to homotopy and such homotopy $h$ is given:
\[
	\xymatrix{
		 & X \ar[dl] \ar[dr] & \\
		E_* \ar[rr]^f & & E'_* \\
	}
\]
where $X\to E_*$ and $X\to E'_*$ are the framings.
Two $X$-framed fiberwise spaces are said to be \textit{$X$-framed equivalent} if there exists an $X$-framed fiberwise map between them. 

Similarly, an ex-fibration $E\to B$ over a path-connected pointed space $B$ is said to be \textit{$X$-framed} if a pointed homotopy equivalence from $X$ to the fiber $E_*$ over the basepoint is given.
The maps and equivalences between them are also similarly defined.
\end{dfn}
\begin{rem}
By Dold's Theorem \cite[Theorem 2.6]{May75}, $X$-framed equivalence is an equivalence relation between $X$-framed Hurewicz or ex-fibrations.
\end{rem}
Let $B$ be a path-connected pointed space.
For a space $X$ with the homotopy type of a CW complex, define a set $\mathcal{E}X(B)_0$ as the set of $X$-framed equivalence classes of $X$-framed Hurewicz fibrations over $B$.
Similarly, for a pointed space $X$ pointed homotopy equivalent to a CW complex, define a set $\mathcal{E}_0X(B)_0$ as the set of $X$-framed equivalence classes of $X$-framed ex-fibrations over $B$.
Let $\varphi :B\to B'$ be a pointed map.
Then the pullback of fibrations by $\varphi$ induces maps $\mathcal{E}X(B')_0\to \mathcal{E}X(B)_0$ and $\mathcal{E}_0X(B')_0\to \mathcal{E}_0X(B)_0$.
This map is determined by a pointed homotopy class of $\varphi$.
This follows from the next lemma, which is proved similarly to \cite[Lemma 2.4]{May75}.
\begin{lem}\label{lem_relativelifting}
Let $E$ and $E'$ be Hurewicz fibrations over $B$ and $B'$, and $A\subset B$ be a closed subset with homotopy extension property.
Then, for a homotopy $\bar h:I\times B\to B'$ and a map $h:(I\times E_A)\cup(\{0\}\times E)\to E'$ covering $\bar h|_{(I\times A)\cup(\{0\}\times B)}$, there exists an extension $h':I\times E\to E'$ of $h$ covering $\bar h$.

For ex-fibrations $E$ and $E'$, the similar assertion holds.
\end{lem} 
The following classification theorem is also shown similarly to the unframed versions \cite[Theorem 9.2]{May75}.
\begin{thm}[The classification theorem for framed fibrations]
The following statements hold.
\begin{enumerate}
\item
Let $X$ be a space with the homotopy type of a CW complex.
Then there exists an $X$-framed Hurewicz fibration $EX$ over the classifying space $BFX$ of the topological monoid $FX$ of self homotopy equivalences on $X$ such that the map $\Phi :[B,BFX]_0\to \mathcal{E}X(B)_0$ given by the pullback $[\varphi]\mapsto [\varphi ^*EX]$ is an isomorphism for any pointed CW complex $B$.
\item
Let $X$ be a pointed space pointed homotopy equivalent to a CW complex.
Then there exists an $X$-framed ex-fibration $EX$ over the classifying space $BHX$ of the topological monoid $HX$ of pointed self homotopy equivalences on $X$ such that the map $\Phi :[B,BHX]_0\to \mathcal{E}_0X(B)_0$ given by the pullback $[\varphi]\mapsto [\varphi ^*EX]$ is an isomorphism for any pointed CW complex $B$.
\end{enumerate}
\end{thm}
\begin{proof}
(1)
Define a map $\Psi:\mathcal{E}X(B)_0\to [B,BFX]_0$ by the following procedure.
For a $X$-framed Hurewicz fibration $E\to B$, we have the associated principal fibration
\[
	PE=\{ \, u:X\to E \, |\, u \mathrm{\, is \, a \, homotopy \, equivalence\, between \,}X \mathrm{\, and \, some \, fiber} \, \}\longrightarrow B
\]
with basepoint given by the framing $X\to E$ and the following diagram of principal quasi-fibrations.
\[
	\xymatrix{
		PE \ar[d] & B(PE,FX,FX) \ar[l]_-{\simeq} \ar[d] \ar[r] & B(*,FX,FX)=EFX \ar[d] \\
		B \ar@/_2pc/[rr]^-{\psi} & B(PE,FX,*) \ar[l]_-{\simeq} \ar[r] & B(*,FX,*)=BFX
	}
\]
The horizontal arrows are natural maps and the map $\psi$ is the composite of the homotopy inverse of $B(PE,FX,*)\to B$ and the map $B(PE,FX,*)\to BFX$.
Define $\Psi (E)=[\psi ]$.
Since our framed maps preserve framing only up to homotopy, the well-definedness of $\Psi$ is not so trivial as in May's original proof.
More precisely, for an $X$-framed fibrewise map $(f,h):E\to E'$, the induced map $B(Pf,{\rm id}_{FX},*):B(PE,FX,*)\to B(PE',FX,*)$ is not necessarily a pointed map.
In order to verify $\Psi(E)=\Psi(E')$, it is sufficient to consider the following commutative diagram:
\begin{align*}
	\xymatrix{
		 & B & \\
		B(PE,FX,*) \ar[ur]^-{\simeq} \ar[dr] & B(PE,FX,*)\vee I \ar[u]^-{\simeq} \ar[l]_-{\simeq} \ar[r]^-{f\vee h} \ar[d] & B(PE',FX,*) \ar[ul]_-{\simeq} \ar[dl] \\
		 & B(*FX,*), &
	}
\end{align*}
where $B(PE,FX,*)\vee I$ is considered as a pointed space with base point $1\in I$ and $f\vee h$ is the map defined by $B(Pf,{\rm id}_{FX},*)$ on $B(PE,FX,*)$ and $h:I\to PE'$ on $I$ with $h(0)=(f_*\circ(\mathrm{the\, framing\, of\,}E))\in PE'$ and $h(1)=(\mathrm{the\, framing\, of\,}E')\in PE'$.
Once the well-definedness is established, the theorem follows from the completely same argument as in May's proof.

The assertion (2) also follows from the analogous argument.
\end{proof}
\begin{rem}
In the corollaries of \cite[Theorem 9.2]{May75}, the fibers are assumed to be compact.
But this assumption is not needed as remarked in the addenda of \cite{May75} and \cite[Lemma 1.1]{May80}.
\end{rem}
\begin{rem}
Similarly to \cite[Theorem 9.2]{May75}, one can also show the ``framed fiber bundle'' version.
\end{rem}

\section{Quasi-category of fiberwise $A_n$-spaces}\label{qcatAn}
In this section, we recall the definition of fiberwise $A_n$-spaces and introduce the framed version.
\begin{dfn}
Let $E$ be an ex-fibration over $B$.
A family of fiberwise maps $\{m_i:\mathcal{K}_i\times E^{\times _Bi}\to E\}_{i=2}^n$ over $B$ is said to be a \textit{fiberwise $A_n$-form} on $E$ if it restricts to an $A_n$-form $\{m_i:\mathcal{K}_i\times (E_b)^i\to E_b\}_{i=2}^n$ on each fiber $E_b$.
An ex-fibration equipped with a fiberwise $A_n$-form on it is called a \textit{fiberwise $A_n$-space}.
Fiberwise $A_n$-maps, their fiberwise homotopies, fiberwise $A_n$-homomorphisms and canonical composition are defined similarly.
\end{dfn}
We note that every $A_n$-space can be considered as a fiberwise $A_n$-space over a point.

Boardman and Vogt \cite{BV73} constructed the quasi-categories (the term used in \cite{Joy02}), which they called restricted Kan complexes, consisting of their ``$W\mathfrak{B}$-spaces''.
\begin{dfn}
Let $\mathcal{R}=\{\mathcal{R}_i\}_{i=0}^\infty$ be a simplicial class.
Then $\mathcal{R}$ is said to be a \textit{quasi-category} if the following \textit{restricted Kan condition} holds for $i\geq 2$:
for any $i-1$-simplices $x_0,\cdots ,x_{j-1},x_{j+1},\cdots ,x_i\in \mathcal{R}_{i-1}$, $0<j<i$, such that $d^{\ell-1}x_{k}=d^{k}x_{\ell}$ for $0\leq k<\ell \leq i$ and $k,\ell \not=j$, there exists an $i$-simplex $x\in \mathcal{R}_i$ such that $d^kx=x_k$ for $k\not=j$.
\par
For a quasi-category $\mathcal{R}$, the \textit{homotopy category} $\mathrm{ho}\mathcal{R}$ of $\mathcal{R}$ is the ordinary category whose objects are the elements in $\mathcal{R}_0$ and the morphisms between $a,b\in \mathcal{R}_0$ are the simplicial homotopy classes of edges $f\in \mathcal{R}_1$ with $d^1f=a$ and $d^0f=b$, where the composite of $[f]:a\to b$ and $[g]:b\to c$ is defined to be such a map $[h]:a\to c$ as there exists a 2-simplex $\sigma \in \mathcal{R}_2$ with $d^2\sigma =f$, $d^0\sigma =g$ and $d^1\sigma =h$.
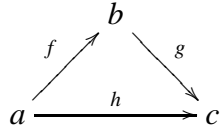
\begin{align*}
	\xymatrix{
		& b \ar[dr]^-{g} & \\
		a \ar[ur]^-{f} \ar[rr]^-{h} & & c
	}
\end{align*}
\end{dfn}
\begin{rem}
The identity on $a$ in $\mathrm{ho}\mathcal{R}$ is represented by $s^0a\in \mathcal{R}_1$.
\end{rem}
Now we recall the terminology and results of \cite{BV73} arranged for fiberwise $A_n$-spaces as explained in \cite[Section 3 and 4]{Tsu12a}.
But now we need the pointed version.
\begin{enumerate}
\item
Fiberwise $A_n$-spaces and fiberwise $Q^nh\mathfrak{A}$-maps.

Let $\mathfrak{A}(n,1)$ be a point for $n\geq 0$ and then $\mathfrak{A}$ has the unique \textit{based monochrome PRO} \cite[Section V.2]{BV73} structure.
Applying the \textit{$W''$-construction} \cite[Section V.3]{BV73} to $\mathfrak{A}$, we obtain the associahedron $(W''\mathfrak{A})(n,1)\cong \mathcal{K}_n$ for each $n\geq 2$ and $W''\mathfrak{A}(0,1)$ and $W''\mathfrak{A}(1,1)$ are one point spaces, where the boundary map $\partial _k(r,s):\mathcal{K}_r\times \mathcal{K}_s\to \mathcal{K}_{r+s-1}$ and the degeneracy map $s_k:\mathcal{K}_i\to \mathcal{K}_{i-1}$ on associahedra correspond to the composite $s$-ary tree on the $k$-th twig and the composite the stump on the $k$-th twig in $W''\mathfrak{A}$, respectively.
We denote the based PRO-subcategory of $W''\mathfrak{A}$ generated by $(W''\mathfrak{A})(i,1)$ for $i\leq n$ by $Q^nW''\mathfrak{A}$.
Then, fiberwise $Q^nW''\mathfrak{A}$-spaces, defined similarly to \cite[Definition 5.1]{BV73}, and our fiberwise $A_n$-spaces coincide.
The linear category $\mathfrak{L}_n=\{0\to 1\to \cdots \to n\}$ is regarded as a PRO colored by the set $\{0,1,\cdots ,n\}$.
Consider the Boardman-Vogt tensor product \cite[Section II.3]{BV73} $\mathfrak{A}\otimes \mathfrak{L}_n$ of $\mathfrak{A}$ and $\mathfrak{L}_n$, which is a based PRO colored by $\{0,1,\cdots ,n\}$.
Similarly to $Q^nW''\mathfrak{A}$, let $Q^nHW''(\mathfrak{A}\otimes \mathfrak{L}_m)$ be the homogeneous PRO-subcategory of $HW''(\mathfrak{A}\otimes \mathfrak{L}_m)$ generated by $HW''(\mathfrak{A}\otimes \mathfrak{L}_m)(i^r,k)$ ($i^r:\{1,\ldots,r\}\to\{0,1,\ldots,m\}$ denotes the constant function with value $i$) with $r\leq n$.
We call a fiberwise $Q^nHW''(\mathfrak{A}\otimes \mathfrak{L}_1)$-space a \textit{fiberwise $Q^nh\mathfrak{A}$-map}.
Further, if its underlying map is a fiberwise pointed homotopy equivalence, it is said to be a \textit{fiberwise $Q^nh\mathfrak{A}$-equivalence}.
Similarly, we use the terms \textit{$Q^nh\mathfrak{A}$-map} and \textit{$Q^nh\mathfrak{A}$-equivalence} between $A_n$-spaces.
\item
Correspondence of fiberwise $A_n$-maps and fiberwise $Q^nh\mathfrak{A}$-maps.

In a similar manner to \cite[Definition 4.1]{BV73}, one can define the based PRO subcategory $LW''(\mathfrak{A}\otimes \mathfrak{L}_1)\subset HW''(\mathfrak{A}\otimes \mathfrak{L}_1)$ consisting of level-trees.
The PRO subcategory $Q^nLW''(\mathfrak{A}\otimes \mathfrak{L}_1)\subset LW''(\mathfrak{A}\otimes \mathfrak{L}_1)$ generated by $LW''(\mathfrak{A}\otimes \mathfrak{L}_1)(i^r,k)$ with $r\leq n$ is a deformation retract of $Q^nHW''(\mathfrak{A}\otimes \mathfrak{L}_1)$ as a based PRO subcategory by a similar argument to the proof of \cite[Proposition 4.6]{BV73}.
Then there is a natural homeomorphism $LW''(\mathfrak{A}\otimes \mathfrak{L}_1)(0^r,1)\cong \mathcal{J}_r$ for each $r$ which has the compatibility about boundary and degeneracy maps similar to the case of associahedra explained above.
Thus fibrewise $Q^nLW''(\mathfrak{A}\otimes \mathfrak{L}_1)$-spaces just correspond to fiberwise $A_n$-maps.
This observation implies that fiberwise $A_n$-maps and fiberwise $Q^nh\mathfrak{A}$-maps correspond one-to-one up to fiberwise homotopy preserving the multiplicative structures.
\item
Quasi-category of fiberwise $A_n$-spaces.

Quite analogously to \cite[Theorem 5.23]{BV73}, the sequence $\mathcal{R}^n(B)=\{\mathcal{R}_i^n(B)\}_{i=0}^\infty$ of classes such that $\mathcal{R}_i^n(B)$ consists of all fiberwise $Q^nW''(\mathfrak{A}\otimes \mathfrak{L}_i)$-spaces over the base $B$ is a quasi-category.
The boundary map $d^k:\mathcal{R}^n_i(B)\to \mathcal{R}^n_{i-1}(B)$ is induced from the $k$-th injection $\mathfrak{L}_{i-1}\to \mathfrak{L}_i$ and the degeneracy map $s^k:\mathcal{R}^n_i(B)\to \mathcal{R}^n_{i+1}(B)$ is induced from the $k$-th surjection $\mathfrak{L}_{i+1}\to \mathfrak{L}_i$.
Every map $B'\to B$ induces a simplicial map $\mathcal{R}^n(B)\to \mathcal{R}^n(B')$ defined by pullback.
\item
Composition of fiberwise $Q^nh\mathfrak{A}$-maps.

The color-forgetting map $\mathcal{J}_n\cong LW''(\mathfrak{A}\otimes \mathfrak{L}_1)(0^n,1)\hookrightarrow HW''(\mathfrak{A}\otimes \mathfrak{L}_1)(0^n,1)\to W''\mathfrak{A}(n,1)\cong \mathcal{K}_n$ induced from the map $\mathfrak{L}_1\to \mathfrak{L}_0$ satisfies the same condition for $\pi _n$ in Section \ref{anspaces}.
For a simplex $\sigma \in \mathcal{R}^n_i(B)$ and a fiberwise $A_n$-homomorphism $g:E_0'\to E_0=\sigma|_{Q^nHW''(\mathfrak{A}\otimes\{0\})}$, define the \textit{canonical composite} $\sigma \circ g\in \mathcal{R}^n_i(B)$ by
\begin{align*}
	(\sigma \circ g)(\rho )(x_1,\cdots ,x_j)=\left\{
	\begin{array}{ll}
		E_0'(\rho )(x_1,\cdots ,x_j) & (\rho \in Q^nHW''(\mathfrak{A}\otimes \mathfrak{L}_i)(0^j,0))\\
		\sigma (\rho )(g(x_1),\cdots ,g(x_i)) & (\rho \in Q^nHW''(\mathfrak{A}\otimes \mathfrak{L}_i)(0^j,\ell ),\ell \geq 1)\\
		\sigma (\rho )(x_1,\cdots ,x_i) & (\rho \in Q^nHW''(\mathfrak{A}\otimes \mathfrak{L}_i)(k^j,\ell ),k \geq 1),
	\end{array}
	\right.
\end{align*}
where $E_0'\in \mathcal{R}^n_0(B)$ denotes the $Q^nW''\mathfrak{A}$-structure of $E'_0$.
Similarly, for a fiberwise $A_n$-homomorphism $h:E_n=\sigma|_{Q^nHW''(\mathfrak{A}\otimes\{n\})}\to E_n'$, the canonical composite $h\circ \sigma \in \mathcal{R}^n_i(B)$ is defined.
Then, for a fiberwise $A_n$-homomorphism $f:E\to E'$, the fiberwise $Q^nh\mathfrak{A}$-map $f\circ (s^0E)=(s^0E')\circ f$ corresponds to the canonical $A_n$-form of $f$.
Using this, the canonical composite is compatible with the composite in the homotopy category $\mathrm{ho}\mathcal{R}^n(B)$ as follows:
for a fiberwise $A_n$-homomorphism $g:E\to E'$ and a fiberwise $Q^nh\mathfrak{A}$-map $f:E'\to E''$, the 2-simplex $\sigma :=(s^0f)\circ g\in \mathcal{R}^n_2(B)$ satisfies $d^2\sigma =(s^0E')\circ g$, $d^0\sigma =f$ and $d^1\sigma =f\circ g$.
\begin{align*}
	\xymatrix{
		&  E'\ar[dr]^-{f} & \\
		E \ar[ur]^-{(s^0E')\circ g} \ar[rr]^-{f\circ g} & & E'' \\
	}
\end{align*}
Similarly, for a fiberwise $Q^nh\mathfrak{A}$-map $f:E\to E'$ and a fiberwise $A_n$-homomorphism $h:E'\to E''$, the 2-simplex $\sigma '=h\circ (s^1f)\in \mathcal{R}^n_2(B)$ satisfies $d^2\sigma '=f$, $d^0\sigma '=h\circ (s^0E')$ and $d^1\sigma '=h\circ f$.
\item
Fiberwise homotopy invariance.

Analogously to \cite[Theorem 5.25]{BV73}, if $E$ is a fiberwise $A_n$-space and $f:E'\to E$ is a fiberwise pointed homotopy equivalence between ex-fibrations, then there exists a fiberwise $A_n$-form on $E'$ such that $f$ admits a fiberwise $A_n$-form.
Also, we obtain the following as the counterpart of \cite[Theorem 5.24]{BV73}:
for a fiberwise $Q^nh\mathcal{A}$-equivalence $f:E\to E'$, the fiberwise homotopy inverse $g$ of the underlying fiberwise pointed map $f$ also admits the structure of a fiberwise $Q^nh\mathfrak{A}$-map which represents the inverse of $f$ in the homotopy category $\mathrm{ho}\mathcal{R}^n(B)$.
Similarly to \cite[Lemma 5.7]{BV73}, for fiberwise $A_n$-forms $\{m_i\}$ and $\{m_i'\}$ on a ex-fibration $E$, $\{m_i\}$ and $\{m_i'\}$ are homotopic as fiberwise $A_n$-forms on $E$ if and only if the identity map $\mathrm{id}_E:E\to E$ admits a fiberwise $A_n$-form as a fiberwise $A_n$-map $(E,\{m_i\})\to (E,\{m_i'\})$.
\item
Fiberwise localization.

Let $E$ be a fiberwise $A_n$-space over $B$ whose fibers are path-connected and $\mathcal{P}$ a set of primes.
Since the fibers of $E$ are nilpotent, there exists a fiberwise $\mathcal{P}$-localization $\ell :E\to \bar E$ \cite{May80}.
By an analogous argument to the proof of Proposition \ref{localizationofA_n}, there is a fiberwise $Q^nh\mathfrak{A}$-map $\lambda \in \mathcal{R}^n_1(B)$ such that $d^1\lambda =E$ and the underlying map of $\lambda$ is $\ell $.
Moreover, for a fiberwise $Q^nh\mathfrak{A}$-map $f:E\to E'$ such that the fibers of $E'$ are $\mathcal{P}$-local, then there exists a 2-simplex $\sigma \in \mathcal{R}^n_2(B)$ such that $d^2\sigma =\lambda$ and $d^1\sigma =f$ because the induced map $(\ell ^i)^*:\Map _B^B(\bar E^i,E')\to \Map ^B_B(E^i,E')$ between the space of pointed fiberwise maps is a weak homotopy equivalent.
This universal property implies that the fiberwise $\mathcal{P}$-localization as a fiberwise $A_n$-space is unique up to fiberwise $A_n$-equivalence.
\end{enumerate}

Now we consider the framed version.
\begin{dfn}
A fiberwise $A_n$-space $E\to B$ over a path-connected pointed space $B$ with every fiber $A_n$-equivalent to an $A_n$-space $G$ is said to be $G$\textit{-framed} if a $Q^nh\mathfrak{A}$-equivalence from $G$ to the fiber $E_*$ over the basepoint is given, which we call the \textit{framing} of $E$.
For $G$-framed fiberwise $A_n$-spaces $E$ and $E'$ over $B$, a fiberwise $Q^nh\mathfrak{A}$-map $f:E\to E'$ is said to be $G$\textit{-framed} if a 2-simplex $\sigma \in \mathcal{R}_2^n(*)$ is given and satisfies the following conditions:
\begin{enumerate}
\item $d^2\sigma$ is the framing $G\to E_*$ of $E$,
\item $d^1\sigma$ is the framing $G\to E'_*$ of $E$,
\item $d^0\sigma$ is the restriction $f_*:E_*\to E'_*$ of $f$.
\end{enumerate}
Two $G$-framed fiberwise $A_n$-spaces are said to be $G$\textit{-framed equivalent} if there exists a $G$-framed fiberwise $Q^nh\mathfrak{A}$-map between them. 
In particular, a $G$-framed fiberwise $A_n$-space over $B$ is said to be \textit{$A_n$-trivial} if it is $G$-framed equivalent to the product $B\times G$ with framing $G\cong *\times G\subset B\times G$.
\end{dfn}
The relation $G$-framed equivalence is an equivalence relation between fiberwise $A_n$-spaces by Dold's theorem and the property of fiberwise $A_n$-spaces stated above.

The following lemma will be used in the next section.
It is shown by the analogous argument to the proof of the lifting theorem \cite[Theorem 3.17]{BV73}.
\begin{lem}\label{homotopyextensioninR_2}
Let $B$ be a space.
Consider a 2-cell $\sigma\in\mathcal{R}^n_2(B)$ and 1-cells $\Sigma_2,\Sigma_0\in\mathcal{R}^n_1(I\times B)$.
If the conditions $\Sigma_2|_{\{0\}\times B}=d^2\sigma$, $\Sigma_0|_{\{0\}\times B}=d^0\sigma$ and $d^0\Sigma_2=d^1\Sigma_0$ are satisfied, then there exists a 2-cell $\Sigma\in\mathcal{R}^n_2(I\times B)$ such that $\Sigma|_{\{0\}\times B}=\sigma$, $d^2\Sigma=\Sigma_2$ and $d^0\Sigma=\Sigma_0$.
\end{lem}

\section{Classification theorem for framed fiberwise $A_n$-spaces}\label{classificationsection}
The author has shown the classification theorem for (unpointed and unframed) fiberwise $A_n$-spaces \cite[Theorem 5.7]{Tsu12a}, which is used to show the finiteness of fiberwise $A_n$-types of adjoint group bundles (see Section \ref{gauge}).
In this section, we show the classification theorem for framed fiberwise $A_n$-spaces.

First, we arrange the classification theorem for our fiberwise $A_n$-spaces.
Let $G$ be an $A_n$-space.
Denote the set of equivalence classes of fiberwise $A_n$-spaces over a path-connected pointed space $B$ whose fibers are $A_n$-equivalent to $G$ by $\mathcal{E}^{A_n}G(B)$.
For an ex-fibration $E$ over $B$ whose fibers are pointed homotopy equivalent to $G$, define a space $M_n[E]$ by
\begin{align*}
	M_n[E]=\coprod_{b\in B}\left. \left\{\, \{m_i\} \in \prod_{i=2}^n \Map(\mathcal{K}_i\times E_b^{i},E_b)\, \right| 
	\begin{array}{l}
		\{m_i\}:{\rm an\,}A_n{\rm \mathchar`-form\, of\,}E_b{\rm \, such\, that\,} \\
		E_b{\rm \,and\,}G{\rm \,are\,}A_n{\rm \mathchar`-equivalent} \\
	\end{array}
	\right\}
\end{align*}
as a set, which is topologized as a subspace of the fiber product of appropriate fiberwise mapping spaces over $B$. 
There is a sequence of natural projections 
\begin{align*}
	M_n[E]\longrightarrow M_{n-1}[E]\longrightarrow \cdots \longrightarrow M_2[E]\longrightarrow M_1[E]=B,
\end{align*}
each of which is a Hurewicz fibration because of the homotopy extension property of the inclusion $\partial \mathcal{K}_i\hookrightarrow \mathcal{K}_i$ and the fiberwise well-pointedness of $E$.
Further, by an easy observation, we obtain the homotopy fiber sequence
\begin{align*}
	\Omega _0^{n-2}\Map _0(G^{\wedge n},G)\longrightarrow M_n[E]\longrightarrow M_{n-1}[E],
\end{align*}
where $G^{\wedge n}$ denotes the $n$-fold smash product $G\wedge \cdots \wedge G$, $\Map _0(X,Y)$ the space of all pointed maps $X\to Y$ and $\Omega ^m_0X$ the space of all pointed maps $S^m\to X$ homotopic to the constant map.
By the exponential law, each fiberwise $A_n$-form on $E$ corresponds to a section of the projection $M_n[E]\to B$.
The pullback $E_n[E]$ of $E$ by the projection $M_n[E]\to B$ has a natural fiberwise $A_n$-form such that the restricted $A_n$-form of the fiber over a point $\{m_i\}\in M_n[E]$ is $\{m_i\}$.
Let $E_1(G)\to M_1(G)=BHG$ be the universal ex-fibration with the fibers pointed homotopy equivalent to $G$.
Let us denote $M_n(G)=M_n[E_1(G)]$ and $E_n(G)=E_n[E_1(G)]$.

Using the properties of fiberwise $A_n$-spaces explained in the previous section, one can prove the following theorem by a completely parallel argument in \cite[Section 5]{Tsu12a}.
Denote the free homotopy set between spaces $X$ and $Y$ by $[X,Y]$. 
\begin{thm}[The classification theorem for fiberwise $A_n$-spaces]\label{unframedclassification}
Let $n$ be a finite positive integer and $G$ a well-pointed $A_n$-space of the homotopy type of a CW complex.
Then, there exists a fiberwise $A_n$-space $E_n(G)\to M_n(G)$ with fibers $A_n$-equivalent to $G$ such that the map $[B,M_n(G)]\to \mathcal{E}^{A_n}G(B)$ defined by the correspondence $[f]\mapsto [f^*E_n(G)]$ is bijective for any well-pointed space $B$ of the homotopy type of a connected CW complex.
\end{thm}
Now, we construct the `associated principal fibration' for $E_n[E]\to M_n[E]$.
Again, let $G$ be an $A_n$-space.
For an ex-fibration $E$ over $B$ whose fibers are pointed homotopy equivalent to $G$, define a space $C_n[E]$ as
\[
	C_n[E]=\coprod _{b\in M_n[E]}\left\{\, f:Q^nh\mathfrak{A} \mathrm{\mathchar`- equivalence},d^1f=G,d^0f=(E_n[E])_b \right\}.
\]
In particular, we denote $C_n(G)=C_n[E_1(G)]$.
Then there exists a homotopy fibration
\[
	F_{A_n}G\longrightarrow C_n[E]\longrightarrow M_n[E],
\]
where $F_{A_n}G$ is the space consisting of self $Q^nh\mathfrak{A}$-equivalences on $G$.
Similarly to \cite[Proposition 5.8]{Tsu12a}, the fiber of the Hurewicz fibration $C_n[E]\to C_{n-1}[E]$ is contractible.
Then $C_n(G)$ is weakly contractible since $C_1(G)=EHG$ is weakly contractible.

For an $A_n$-space $G$, denote the set of equivalence classes of $G$-framed fiberwise $A_n$-spaces over a path-connected pointed space $B$ by $\mathcal{E}^{A_n}G(B)_0$.
Using Lemma \ref{lem_relativelifting}, one can show that a pointed map $f:B\to B'$ induces the map $f^*:\mathcal{E}^{A_n}G(B')_0\to \mathcal{E}^{A_n}G(B)_0$ which depends only on the pointed homotopy class of $f$.
It is our aim of this section to prove the following classification theorem.
\begin{thm}[The classification theorem for framed fiberwise $A_n$-spaces]\label{thm_cls_fr_fw_A_n}
Let $n$ be a finite positive integer and $G$ an $A_n$-space of the homotopy type of a CW complex.
Then, there exists a $G$-framed fiberwise space $E_n(G)\to M_n(G)$ such that the map $[B,M_n(G)]_0\to \mathcal{E}^{A_n}G(B)_0$ defined by the correspondence $[f]\mapsto [f^*E_n(G)]$ is bijective for any pointed space $B$ pointed homotopy equivalent to a connected CW complex.
\end{thm}
\begin{proof}
We fix a $G$-framing of the universal fiberwise $A_n$-space $E_n(G)$.
First, we see that the map $[B,M_n(G)]_0\to \mathcal{E}^{A_n}G(B)_0$ is surjective.
For a $G$-framed fiberwise $A_n$-space $E\to B$, by Theorem \ref{unframedclassification}, there are maps $f:E\to E_n(G)$ and $\bar f:B\to M_n(G)$ such that $f$ covers $\bar f$ and $f$ induces a fiberwise $Q^nh\mathfrak{A}$-equivalence $E\to \bar f^*M_n(G)$.
Take a 2-cell $\sigma\in\mathcal{R}^n_2(*)$ such that $d^2\sigma:E_*\to G$ represents the inverse of the framing $G\to E_*$ in the homotopy category and $d^0\sigma=f_*:E_*\to E_n(G)_*$ as above.
There is a 1-cell $\Sigma^0\in\mathcal{R}^n_1(I)$ defined by a path $I\to PE_n(G)$ from $d^1\sigma$ to the framing of $E_n(G)$ since $PE_n(G)$ is weakly contractible.
Denote the stationary homotopy on $d^2\sigma$ by $\Sigma^2\in\mathcal{R}^n_1(I)$.
Then by Lemma \ref{homotopyextensioninR_2}, there exists a 2-sell $\Sigma\in\mathcal{R}^n_2(I)$ such that $\Sigma|_{\{0\}}=\sigma$, $d^2\Sigma=\Sigma_2$ and $d^0\Sigma=\Sigma_0$.
The homotopy $d^1\Sigma$ extends to a homotopy $F:I\times E\to E_n(G)$ from $f$ covering some homotopy $\bar F:I\times B\to M_n(G)$ from $\bar f$.
Denote $g=F|_{\{1\}\times E}$ and $\bar g=\bar F|_{\{1\}\times B}$.
Since $d^2\Sigma|_{\{1\}}:E^*\to G$ represents the inverse of the framing, $d^0\Sigma|_{\{1\}}:G\to E_n(G)$ is the framing and $d^1\Sigma|_{\{1\}}=g_*$, the map $g$ induces a $G$-framed fiberwise $Q^nh\mathfrak{A}$-map $E\to\bar g^*E_n(G)$.
Thus we obtain the equality $[\bar g^*E_n(G)]=[E]$ in $\mathcal{E}^{A_n}G(B)_0$.

The injectivity of the map $[B;M_n(G)]_0\to\mathcal{E}_0^{A_n}G(B)_0$ is proved similarly to \cite[Proposition 5.6]{Tsu12a}.
The counterpart of \cite[Lemma 5.3]{Tsu12a} is stated as follows.
\end{proof}
\begin{lem}
Let $E\to B$ be an ex-fibration and $\{m_i\}_{i=2}^n$ and $\{m_i'\}_{i=2}^n$ be fiberwise $A_n$-forms on $E$ which restrict to the same $A_n$-form on the fiber over the basepoint.
Then the identity map $E\to E$ is a framed fiberwise $Q^nh\mathfrak{A}$-map $(E,\{m_i\})\to (E,\{m_i'\})$ if and only if $\{m_i\}$ and $\{m_i'\}$ are homotopic as sections of $M_n[E]\to B$ through a homotopy which is constant on the basepoint.
\end{lem}
Similarly to \cite[Proposition 6.3]{Tsu12a}, we have the following proposition for fiberwise localizations.
\begin{prp}\label{A_nfiberwiselocalization}
Let $G$ be a path-connected $A_n$-space.
Denote the classifying map of the fiberwise $\mathcal{P}$-localization of $E_n(G)$ by $\lambda :M_n(G)\to M_n(G_{\mathcal{P}})$ as a framed fiberwise $A_n$-space.
Then for a ($G$-framed) fiberwise $A_n$-space $E$ over $B$ classified by $\alpha :B\to M_n(G)$, the fiberwise $\mathcal{P}$-localization of $E$ is classified by $\lambda \circ \alpha$.
Moreover, if $G$ is homotopy equivalent to a finite complex and $r\geq 2$, then the induced homomorphism $\lambda _*:\pi _r(M_n(G))\to \pi _r(M_n(G_{\mathcal{P}}))$ is a $\mathcal{P}$-localization of the abelian group $\pi _r(M_n(G))$.
\end{prp}
\begin{rem}
The finiteness condition for $G$ is used as follows.
May \cite[Theorem 4.1]{May80} proved that $\lambda _*:\pi _r(M_1(G))\to \pi _r(M_1(G_{\mathcal{P}}))$ is a $\mathcal{P}$-localization for $r\geq 2$ when $G$ is a finite complex.
By the fibration
\[
	\Omega ^{n-2}\Map _0(G^{\wedge n},G)\longrightarrow M_n(G)\longrightarrow M_{n-1}(G)
\]
and the fact that $\Map _0(G^{\wedge n},G)\to \Map _0(G^{\wedge n}_{\mathcal{P}},G_{\mathcal{P}})$ $\mathcal{P}$-localizes each connected components for a connected finite complex $G$ \cite[Theorem 3.11]{HMR75}, one can inductively show that $\lambda _*:\pi _r(M_n(G))\to\pi _r(M_n(G_{\mathcal{P}}))$ is a $\mathcal{P}$-localization of the abelian group $\pi _r(M_n(G))$ if $G$ is homotopy equivalent to a finite complex and $r\geq 2$.
\end{rem}

\section{An application to $A_n$-types of gauge groups}\label{gauge}
Let us recall elementary facts on gauge groups.
In this section, we assume that every principal bundle have the basepoint in the fiber over the basepoint of the base space.
For a principal $G$-bundle $P$ over a pointed finite CW complex $B$, the \textit{gauge group} $\mathcal{G}(P)$ of $P$ is the topological group consisting of unpointed $G$-equivariant self maps $P\to P$ covering the identity on $B$, which is topologized by the compact open topology.
Note that $\mathcal{G}$ is a contravariant functor from the category of principal $G$-bundles and bundle maps to the one of topological groups and continuous homomorphisms.
We denote the restriction map on the fiber over the basepoint by $ev:\mathcal{G}(P)\to \mathcal{G}(P|_*)\cong G$ and the kernel of $ev:\mathcal{G}(P)\to G$ by $\mathcal{G}_0(P)$.
In other words, the subgroup $\mathcal{G}_0(P)\subset \mathcal{G}(P)$ consists of pointed $G$-equivariant self maps $P\to P$ covering the identity map on $B$.

Gottlieb \cite{Got72} proved that the classifying space $B\mathcal{G}(P)$ of $\mathcal{G}(P)$ has the homotopy type of the connected component $\Map (B,BG;\alpha )\subset \Map (B,BG)$ of the basepoint free mapping space containing the classifying map $\alpha :B\to BG$ of $P$.
Similarly, $B\mathcal{G}_0(P)$ has the homotopy type of the connected component $\Map _0(B,BG;\alpha )\subset \Map _0(B,BG)$ of the pointed mapping space containing $\alpha$.
Moreover, there exists the homotopy fiber sequence:
\begin{align*}
\mathcal{G}_0(P)\to \mathcal{G}(P)\xrightarrow{ev}G\xrightarrow{\delta }\Map _0(B,BG;\alpha )\to \Map (B,BG;\alpha )\xrightarrow{ev}BG,
\end{align*}
where $ev:\Map (B,BG;\alpha )\to BG$ is the evaluation map on the basepoint.

Take the fiber bundle $\ad P=P\times _GG$ associated to $P$, where the left $G$-action on $G$ is given by conjugation.
We call $\ad P$ the \textit{adjoint group bundle} of $P$.
The adjoint group bundle $\ad P$ is a fiberwise topological group.
Then the space $\varGamma (\ad P)$ of sections is a topological group with the multiplication induced from the fiberwise topological group structure of $\ad P$.
In fact, the gauge group $\mathcal{G}(P)$ is naturally isomorphic to $\varGamma (\ad P)$ as a topological group.
Note that the basepoint of $P$ defines the natural $G$-framing $\ad P|_*\cong G$ which is an isomorphism of topological groups.
\begin{rem}
The finiteness theorem on fiberwise $A_n$-types of adjoint group bundles \cite[Theorem 8.6]{Tsu12a} also holds for framed fiberwise $A_n$-types.
The proof is almost same as the original one except for that we need to use the fibration
\begin{align*}
	\Omega _0^{n-2}\Map _0(G^{\wedge n},G)\longrightarrow M_n(G)\longrightarrow M_{n-1}(G).
\end{align*}
In particular, for an adjoint group bundle $\ad P$ over a sphere $S^r$ for $f\geq 1$, the classifying map in $\pi_r(M_n(G))$ has finite order.
\end{rem}

We shall investigate gauge groups using the localization technique.
Though a localization of a topological group is not a topological group in general, a localization of it has a natural $A_\infty$-form by Proposition \ref{localizationofA_n}.
But, as explained before, it is difficult to handle homotopy pullback of general $A_n$-maps in our setting.
To avoid this difficulty, we replace the Lie group $G$ by a convenient one.
\begin{prp}
Let $G$ be a compact connected Lie group.
For a partition $\mathcal{P}\sqcup \mathcal{P}'$ of the set of primes, there exist a $\mathcal{P}$-localization $G_{\mathcal{P}}$, a $\mathcal{P}'$-localization $G_{\mathcal{P}'}$ and a rationalization $G_{(0)}$ of $G$ such that they are topological groups and a rationalization $\rho :G_{\mathcal{P}}\to G_{(0)}$ and $\rho ':G_{\mathcal{P}'}\to G_{(0)}$ can be taken to be homomorphisms as well as Hurewicz fibrations.
\end{prp}

\begin{proof}
We may take the localized classifying spaces $(BG)_{\mathcal{P}}$, $(BG)_{\mathcal{P}'}$ and $(BG)_{(0)}$ as countable simplicial complexes.
Take rationalizations $\bar \rho :(BG)_{\mathcal{P}}\to (BG)_{(0)}$ and $\bar \rho ':(BG)_{\mathcal{P}'}\to (BG)_{(0)}$ as simplicial maps.
Denote the Milnor's simplicial loop spaces \cite{Mil56} of them by $(G_{\mathcal{P}})_0$, $(G_{\mathcal{P}'})_0$ and $G_{(0)}$, which are topological groups with classifying spaces $(BG)_{\mathcal{P}}$, $(BG)_{\mathcal{P}'}$ and $(BG)_{(0)}$, respectively.
Then $\bar \rho $ and $\bar \rho '$ induce rationalizing homomorphisms $\rho _0:(G_{\mathcal{P}})_0\to G_{(0)}$ and $\rho '_0:(G_{\mathcal{P}'})_0\to G_{(0)}$, respectively.
As remarked in Remark \ref{replace}, $\rho _0$ and $\rho_0'$ can be replaced by Hurewicz fibrations $\rho :G_{\mathcal{P}}\to G_{(0)}$ and $\rho ':G_{\mathcal{P}'}\to G_{(0)}$ which are homomorphisms between topological groups.
\end{proof}

The topological group $\hat G$ obtained as the pullback of the diagram $G_{\mathcal{P}}\stackrel{\rho}{\longrightarrow }G_{(0)}\stackrel{\rho '}{\longleftarrow }G_{\mathcal{P}'}$ is $A_\infty $-equivalent to $G$ and the natural projections $\lambda :\hat G\to G_{\mathcal{P}}$ and $\lambda ':\hat G\to G_{\mathcal{P}}$ are localizing homomorphisms.
Then principal $G$-bundles and principal $\hat G$-bundles correspond one-to-one up to isomorphism and, considering the homotopy types of the classifying spaces, this correspondence preserves the $A_\infty$-types of gauge groups.
If $P$ is a principal $\hat G$-bundle, then the associated bundles $P_{\mathcal{P}}$, $P_{\mathcal{P}'}$ and $P_{(0)}$ induced from the homomorphisms $\hat G\to G_{\mathcal{P}}$, $\hat G\to G_{\mathcal{P}'}$ and $\hat G\to G_{(0)}$ give the corresponding fiberwise localizations $\ad P_{\mathcal{P}}$, $\ad P_{\mathcal{P}'}$ and $\ad P_{(0)}$ of $\ad P$, respectively.

Now we concentrate on our attention to the gauge groups of principal $G$-bundles over the $r$-dimensional sphere $S^r$.
Fix a pointed map $\epsilon :S^{r-1}\to G$.
For an integer $k\in \mathbb{Z}$, we denote the principal $G$-bundle with classifying map $k\epsilon$ by $P_k$.
As easily checked, if the adjoint group bundles $\ad P_k$ and $\ad P_{k'}$ are $G$-framed fiberwise $A_n$-equivalent, then one can deform the framed fiberwise $Q^nh\mathfrak{A}$-equivalence $\ad P_k\to \ad P_{k'}$ to the one $f:\ad P_k\to \ad P_{k'}$ which restricts to the isomorphism of topological groups with canonical $A_n$-form on the fiber over the basepoint since the framings of adjoint group bundles are given by isomorphisms of topological groups.
Moreover, if $\varphi :D^r\to S^r$ is the characteristic map, then we obtain the following strictly commutative diagram of topological groups and $A_n$-maps with respect to canonical composition:
\[
\xymatrix{
\varGamma (\varphi ^*\ad P_k) \ar[r] \ar[d]_-{(\varphi ^*f)_*} & \varGamma ^{\mathrm{id}}((\varphi ^*\ad P_k)|_{S^{r-1}}) \ar[d]^-{(\varphi ^*f)_*} & G \ar[l] \ar@{=}[d] \\
\varGamma (\varphi ^*\ad P_{k'}) \ar[r] & \varGamma ^{\mathrm{id}}((\varphi ^*\ad P_{k'})|_{S^{r-1}}) & G \ar[l] \\
}
\]
where $\varGamma ^{\mathrm{id}}((\varphi ^*\ad P_k)|_{S^{r-1}})$ represents the identity component of the topological group $\varGamma ((\varphi ^*\ad P_k)|_{S^{r-1}})$, $\varGamma (\varphi ^*\ad P_k)\to \varGamma ^{\mathrm{id}}((\varphi ^*\ad P_k)|_{S^{r-1}})$ and $\varGamma (\varphi ^*\ad P_{k'})\to \varGamma ^{\mathrm{id}}((\varphi ^*\ad P_{k'})|_{S^{r-1}})$ are the restrictions on the boundary $S^{r-1}$ of $D^r$, which are Hurewicz fibrations, $G\to \varGamma ^{\mathrm{id}}((\varphi ^*\ad P_k)|_{S^{r-1}})$ and $G\to \varGamma ^{\mathrm{id}}((\varphi ^*\ad P_{k'})|_{S^{r-1}})$ are the inclusions through the framings of the fiber over the basepoint and $\varGamma (\varphi ^*\ad P_k)\to \varGamma (\varphi ^*\ad P_{k'})$ and $\varGamma ((\varphi ^*\ad P_k)|_{S^{r-1}})\to \varGamma ((\varphi ^*\ad P_{k'})|_{S^{r-1}})$ are the induced $A_n$-equivalence from the framed fiberwise $Q^nh\mathfrak{A}$-equivalence $\varphi ^*f:\varphi ^*\ad P_k\to \varphi ^*\ad P_{k'}$.
The gauge groups $\mathcal{G}(P_k)$ and $\mathcal{G}(P_{k'})$ are isomorphic to the pullback of the corresponding horizontal lines.
Since $\varphi ^*P_k$ and $\varphi ^*P_{k'}$ are trivial bundles, the following diagram is equivalent to the above one,
\[
	\xymatrix{
		\Map (D^r,G) \ar[r] \ar@{=}[d] & \Map (S^{r-1},G;0) \ar[d]^-{F} & G \ar[l]_-{\alpha _k} \ar@{=}[d] \\
		\Map (D^r,G) \ar[r]  & \Map (S^{r-1},G;0) & G \ar[l]_-{\alpha _{k'}} \\
	}
\]
where $\Map (S^{r-1},G;0)$ denotes the identity component of the topological group $\Map (S^{r-1},G)$.
The map $\alpha _k:G\to \Map (S^{r-1},G;0)$ is defined as
\begin{align*}
	\alpha _k(g)(s)=(k\epsilon )(s)g(k\epsilon )(s)^{-1}
\end{align*}
for $g\in G$ and $s\in S^{r-1}$.
The right square commutes strictly and the left one up to homotopy of $A_n$-maps.
Denote this homotopy by $H$.
Since $F$ is induced from the framed fiberwise $Q^nh\mathfrak{A}$-equivalence $f$, the canonical composite of $H$ and the evaluation map $\Map(S^{r-1},G;0)\to G$ is the stationary homotopy of the evaluation map $\Map (D^r,G)\to G$.
Now, we give the proof of our main theorem.

\begin{proof}[Proof of Theorem \ref{thm}]
Let $\mathcal{P}$ be the set of primes which divide $N$ and $\mathcal{P}'$ the set of the other primes.
As above, we may assume that the structure group is $\hat G$.
Denote the localizing homomorphisms by $\lambda :\hat G\to G_{\mathcal{P}}$, $\lambda':\hat G\to G_{\mathcal{P}'}$, $\rho :G_{\mathcal{P}}\to G_{(0)}$ and $\rho ':G_{\mathcal{P}'}\to G_{(0)}$, where $\rho\circ\lambda =\rho'\circ\lambda'$.
So, in this proof, we denote the principal $\hat G$-bundle corresponding to the original principal $G$-bundle $P_k$ by the same symbol $P_k$, and so on.
Then, since $(N,k)=(N,k')$, there exists an integer $A\in \mathbb{Z}$ prime to $N$ such that $Ak\equiv k'\, {\rm mod}\, N$.
From Theorem \ref{thm_cls_fr_fw_A_n}, $\ad P_{Ak}$ and $\ad P_{k'}$ are framed fiberwise $Q^nh\mathfrak{A}$-equivalent.
Thus we have the following diagram:
\begin{align}
	\xymatrix{
		\Map (D^r,G_{\mathcal{P}}) \ar[r] \ar@{=}[d] & \Map (S^{r-1},G_{\mathcal{P}};0) \ar[d]^-{A^*} & G_{\mathcal{P}} \ar[l]_-{\alpha _k} \ar@{=}[d] \\
		\Map (D^r,G_{\mathcal{P}}) \ar[r] \ar@{=}[d] & \Map (S^{r-1},G_{\mathcal{P}};0) \ar[d]^-{F} & G_{\mathcal{P}} \ar[l]_-{\alpha _{Ak}} \ar@{=}[d] \\
		\Map (D^r,G_{\mathcal{P}}) \ar[r] & \Map (S^{r-1},G_{\mathcal{P}};0) & G_{\mathcal{P}} \ar[l]_-{\alpha _{k'}} \\
	} \tag{*}
\end{align}
where $F$ is an $A_n$-equivalence induced from the framed fiberwise $Q^nh\mathfrak{A}$-equivalence $\ad P_{Ak}\to \ad P_{k'}$ which restricts to an isomorphism on the fiber over the basepoint, the right squares commute strictly and the left ones up to homotopies such that the canonical composite of them with the evaluation map $\Map(S^{r-1},G_{\mathcal{P}};0)\to G_{\mathcal{P}}$ are stationary homotopy.
We remark that the map $A^*:\Map (S^{r-1},G_{\mathcal{P}};0)\to \Map (S^{r-1},G_{\mathcal{P}};0)$ is an $A_\infty $-equivalence since $A$ is not divided by any prime in $\mathcal{P}$.
Similarly, since the fiberwise $\mathcal{P}'$-localizations of $\ad P_{k}$ and $\ad P_{k'}$ are $A_n$-trivial by Proposition \ref{A_nfiberwiselocalization}, we obtain the following diagram:
\begin{align}
	\xymatrix{
		\Map (D^r,G_{\mathcal{P}'}) \ar[r] \ar@{=}[d] & \Map (S^{r-1},G_{\mathcal{P}'};0) \ar[d]^-{F'} & G_{\mathcal{P}'} \ar[l]_-{\alpha _{k}} \ar@{=}[d] \\
		\Map (D^r,G_{\mathcal{P}'}) \ar[r] & \Map (S^{r-1},G_{\mathcal{P}'};0) & G_{\mathcal{P}'} \ar[l]_-{\alpha _{k'}} \\
	} \tag{**}
\end{align}
where $F'$ is an $A_n$-equivalence, the right square commutes strictly and the left one commutes up to homotopy with the same property as stated above.
Note that the evaluation map $\Map (S^{r-1},G_{(0)};0)\to G_{(0)}$ is a homotopy equivalence because the homotopy fiber $\Omega ^{r-1}_0G_{(0)}$ of it is contractible since $r\geq 2n_\ell$.
Then, since the pullback of the diagram
\begin{align*}
\Map (S^{r-1},G_{\mathcal{P}};0)\xrightarrow{\rho _*}\Map (S^{r-1},G_{(0)};0)\xleftarrow{\rho '_*}\Map (S^{r-1},G_{\mathcal{P}'};0)
\end{align*}
is $\Map (S^{r-1},\hat G;0)$, there exists a lift $\hat F:\Map (S^{r-1},\hat G;0)\to \Map (S^{r-1},\hat G;0)$, which is an $A_n$-equivalence, of $F\circ A^*\circ \lambda _*:\Map (S^{r-1},\hat G;0)\to \Map (S^{r-1},G_{\mathcal{P}};0)$ and $F'\circ \lambda _*':\Map (S^{r-1},\hat G;0)\to \Map (S^{r-1},G_{\mathcal{P}'};0)$ by Theorem \ref{homomorphismpullback}.
This lift is determined uniquely up to homotopy by Theorem \ref{homomorphismpullback2} taking the homotopy between $\rho _*\circ F\circ A^*\circ \lambda _*$ and $\rho _*'\circ F'\circ \lambda _*'$ as the lift of the stationary homotopy of  $ev\circ \rho _*\circ \lambda _*:\Map (S^{r-1},\hat G;0)\to G_{(0)}$ through the homotopy equivalence $ev:\Map (S^{r-1},G_{(0)})\to G_{(0)}$.
Now we have the following diagram:
\[
	\xymatrix{
		\Map (D^r,\hat G) \ar[r] \ar@{=}[d] & \Map (S^{r-1},\hat G;0) \ar[d]^-{\hat F} & \hat G \ar[l]_-{\alpha _{k}} \ar@{=}[d] \\
		\Map (D^r,\hat G) \ar[r] & \Map (S^{r-1},\hat G;0) & \hat G \ar[l]_-{\alpha _{k'}} \\
	}
\]
Again, using Theorem \ref{homomorphismpullback} and \ref{homomorphismpullback2}, this diagram commutes up to homotopy of $A_n$-maps.
Therefore, taking the pullback along the horizontal direction, $\mathcal{G}(P_k)$ and $\mathcal{G}(P_{k'})$ are $A_n$-equivalent by Theorem \ref{homomorphismpullback}. 
\end{proof}

\section{$A_n$-types of gauge groups of principal $\SU(2)$-bundle over $S^4$}\label{section_classification}
In this section, we consider the gauge groups of principal $\SU(2)$-bundles over $S^4$.
Denote the principal $\SU(2)$-bundle over $S^4$ of the second Chern number $k\in\mathbb{Z}$ by $P_k$, where this notation is compatible with the notation in Theorem \ref{thm} for the appropriate generator $\epsilon \in \pi _3(\SU (2))\cong \mathbb{Z}$.
We will denote the fiberwise $p$-localization of $\ad P_k$ by $\ad P_{k,p}$.
Let $a_n^{\mathrm{fw}}$ be the least positive integer such that the adjoint group bundle $\ad P_{a_n^{\rm fw}}$ is $A_n$-trivial.
In this section, we will determine $a_n^{\mathrm{fw}}$ up to power of $2$.
By Proposition \ref{A_nfiberwiselocalization}, $a_n^{\rm fw}$ is the product of the order of the classifying maps of $\ad P_{1,p}$ in $\pi_4(M_n(\SU(2)_{(p)}))\cong\pi_4(M_n(\SU(2)))_{(p)}$ for all prime $p$.
From the result of \cite{CS00}, we already have $a_1^{\rm fw}=12$ and $a_2^{\rm fw}=180$.

By the result of Kishimoto and Kono \cite{KK10}, the adjoint bundle $\ad P_{k,p}$ is $A_n$-trivial if and only if there exists a map $S^4_{(p)}\times\mathbb{H}P^n_{(p)}\to\mathbb{H}P^\infty_{(p)}$ that restricts to the map $(k,i):S^4_{(p)}\vee\mathbb{H}P^n_{(p)}\to\mathbb{H}P^\infty_{(p)}$, where $i:\mathbb{H}P^n_{(p)}\to\mathbb{H}P^\infty_{(p)}$ is the inclusion.

\begin{prp}\label{oddtriviality}
Suppose $p$ is an odd prime.
\begin{enumerate}
\item The adjoint bundle $\ad P_{k,p}$ is $A_{(p-1)/2-1}$-trivial.
\item If the adjoint group bundle $\ad P_{k,p}$ is $A_n$-trivial (suppose nothing if $n=0$), then $\ad P_{kp,p}$ is $A_{n+(p-1)/2}$-trivial.
\end{enumerate}
\end{prp}
\begin{proof}
If $\ad P_{k,p}$ is $A_n$-trivial, as noted above, there is a map
	\begin{align*}
		f:S^4_{(p)}\times\mathbb{H}P^n_{(p)}\cup *\times\mathbb{H}P^\infty_{(p)}\rightarrow\mathbb{H}P^\infty_{(p)}
	\end{align*}
which is an extension of $(k,\mathrm{id}):S^4_{(p)}\vee\mathbb{H}P^\infty_{(p)}\to\mathbb{H}P^\infty_{(p)}$.
It is well-known that there are isomorphisms
	\begin{align*}
		\pi_i(\mathbb{H}P^\infty_{(p)})\cong\pi_{i-1}(S^3_{(p)})\cong\left\{
		\begin{array}{ll}
			\Z/p\Z & (i=2p+1) \\
			0 & (\mathrm{otherwise})
		\end{array}
		\right.
	\end{align*}
for $4<i<4p-1$.
Then, by an easy observation of the obstructions to extending the map $f$, the assertion (1) follows.
Consider the commutative diagram
\begin{align*}
	\xymatrix{
		S^{2p+1}_{(p)} \ar[r] \ar[d]_{p} & S^4_{(p)}\times\mathbb{H}P^{(p-1)/2-1}_{(p)}\cup *\times\mathbb{H}P^\infty_{(p)} \ar[r]^-{f\circ(p\times\mathrm{id})} \ar[d]_-{p\times\mathrm{id}} & \mathbb{H}P^\infty_{(p)} \ar@{=}[d] \\
		S^{2p+1}_{(p)} \ar[r] & S^4_{(p)}\times\mathbb{H}P^{(p-1)/2-1}_{(p)}\cup *\times\mathbb{H}P^\infty_{(p)} \ar[r]^-{f} & \mathbb{H}P^\infty_{(p)},
	}
\end{align*}
where the left horizontal maps are the attaching map of the $p$-local cell of dimension $2p+2$ in $S^4_{(p)}\times\mathbb{H}P^{(p-1)/2}_{(p)}$.
Again, by the above computation of $\pi_i(\mathbb{H}P^\infty_{(p)})$, the obstruction to extending $f\circ(p\times\mathrm{id})$ vanishes and (2) for $n<(p-1)/2$ also follows.

From now, we assume $n\ge(p-1)/2$.
Define a space $X_r$ by the following pushout diagram
	\begin{align*}
		\xymatrix{
			S^4_{(p)}\times\mathbb{H}P^n_{(p)}\cup *\times\mathbb{H}P^\infty_{(p)} \ar[r]^-{f} \ar[d] & \mathbb{H}P^\infty_{(p)} \ar[d]^-{j_r} \\
			S^4_{(p)}\times\mathbb{H}P^{n+r}_{(p)}\cup *\times\mathbb{H}P^\infty_{(p)} \ar[r]^-{f_r} & X_r, 
		}
	\end{align*}
where the left vertical arrow is the inclusion.
There is a natural $p$-local relative cell decomposition $X_r=\mathbb{H}P^\infty_{(p)}\cup e^{4n+8}\cup e^{4n+12}\cup\cdots\cup e^{4n+4r+4}$ and we denote the attaching map of $e^{4n+4r+4}\subset X_r$ by $\varphi_r:S^{4n+4k+3}\to X_{r-1}$.
In the following, we only consider the case $r\le(p-1)/2$

Set generators $u\in H^4(S^4_{(p)};\Z_{(p)})$ and $c\in H^4(\mathbb{H}P^\infty_{(p)};\Z_{(p)})$.
The cohomology ring of $X_r$ is computed as
	\begin{align*}
		H^*(X_r;\Z_{(p)})=\Z_{(p)}[x,y]/(y^2,x^ry),
	\end{align*}
where $j_r^*x=c\in H^4(\mathbb{H}P^\infty_{(p)};\Z_{(p)})$ and $f_r^*y=u\times c^{n+1}\in H^4(S^4_{(p)}\times\mathbb{H}P^\infty_{(p)};\Z_{(p)})$.
Then, by the Serre spectral sequence, one can compute the cohomology of the homotopy fiber $R_r$ of $j_r:\mathbb{H}P^\infty_{(p)}\to X_r$ as
	\begin{align*}
		H^i(R_r;\Z_{(p)})\cong\left\{
		\begin{array}{ll}
			\Z_{(p)} & (i=4n+7\,\mathrm{or}\,4n+4r+6) \\
			0 & (\mathrm{otherwise})
		\end{array}
		\right.
	\end{align*}
for $i<8n+14$ since we assume $n\ge(p-1)/2$ and $r\le(p-1)/2$.
From this isomorphism, there is a map
	\begin{align*}
		R_r\rightarrow K(\Z_{(p)},4n+7)\times K(\Z_{(p)},4n+4r+6)
	\end{align*}
which induces an isomorphism on the cohomology of degree $<4n+2p+6$ if $r<(p-1)/2$.
Thus the homotopy groups are
	\begin{align*}
		\pi_i(R_r)\cong\left\{
		\begin{array}{ll}
			\Z_{(p)} & (i=4n+7\,\mathrm{or}\,4n+4r+6) \\
			0 & (\mathrm{otherwise})
		\end{array}
		\right.
	\end{align*}
for $i<4n+2p+4$.
The above description of $\pi_i(R_r)$ also holds for $r=(p-1)/2,i<4n+2p+4$.
From the homotopy fiber sequence $R_r\to\mathbb{H}P^\infty\to X_r$, we can compute $\pi_i(X_r)$ as follows.
\begin{itemize}
\item
When $4n+8<i<4n+4r+6$, the map $(j_r)_*:\pi_i(\mathbb{H}P^\infty_{(p)})\to\pi_i(X_r)$ is an isomorphism.
\item
When $r<(p-1)/2$ and $4n+4r+7<i<4n+2p+4$, the map $(j_r)_*:\pi_i(\mathbb{H}P^\infty_{(p)})\to\pi_i(X_r)$ is again an isomorphism.
Then, for $r<(p-1)/2$, there is the following short exact sequence
\begin{align*}
	0\rightarrow\pi_{4n+4r+8}(X_{r+1},X_r)\rightarrow\pi_{4n+4r+7}(X_r)\rightarrow\pi_{4n+4r+7}(X_{r+1})\rightarrow 0
\end{align*}
since there is an isomorphism $\pi_{4n+4r+7}(X_{r+1},X_r)=0$, $\pi_{4n+4r+8}(X_{r+1},X_r)\cong\mathbb{Z}_{(p)}$ and $\pi_{4n+4r+8}(X_{r+1})$ is a finite group.
We have a canonical section $\pi_{4n+4r+7}(X_{r+1})\cong\pi_{4n+4r+7}(\mathbb{H}P^\infty_{(p)})\to\pi_{4n+4r+7}(X_r)$ of it and obtain the equality
\begin{align*}
	\pi_{4n+4r+7}(X_r)=(j_r)_*\pi_{4n+4r+7}(\mathbb{H}P^\infty_{(p)})\oplus\Z_{(p)}\{\varphi_{r+1}\}.
\end{align*}
\end{itemize}

We have another generator of the free part of $\pi_{4n+4r+7}(X_r)$: the relative Whitehead product $\theta_r=[(j_r)_*i,\bar{\varphi}_r]\in\pi_{4n+4r+7}(X_r)$ of the inclusion $j_r\circ i:S^4\to X_r$ and the characteristic map $\bar{\varphi}_r:(D^{4n+4r+4},S^{4n+4r+3})\to(X_r,X_{r-1})$ of $e^{4n+4r+4}\subset X_r$ has infinite order and is not divided by $p$ for $r<(p-1)/2$.
This follows from the similar computation on $\pi_i(X_r\cup_{\theta_r}D^{4n+4r+8})$ and the exact sequence
\begin{align*}
	0\rightarrow\pi_{4n+4r+8}(X_r\cup_{\theta_r}D^{4n+4r+8},X_r)\rightarrow\pi_{4n+4r+7}(X_r)\rightarrow\pi_{4n+4r+7}(X_r\cup_{\theta_r}D^{4n+4r+8})\rightarrow 0.
\end{align*}

Consider the next similar pushout diagram
\begin{align*}
		\xymatrix{
			S^4_{(p)}\times\mathbb{H}P^n_{(p)}\cup *\times\mathbb{H}P^\infty_{(p)} \ar[r]^-{f\circ(p\times\mathrm{id})} \ar[d] & \mathbb{H}P^\infty_{(p)} \ar[d]^-{j_r'} \\
			S^4_{(p)}\times\mathbb{H}P^{n+r}_{(p)}\cup *\times\mathbb{H}P^\infty_{(p)} \ar[r]^-{f_r'} & X_r'.
		}
	\end{align*}
and denote the attaching map of the cell $(e')^{4n+4r+4}\subset X_r'$ by $\varphi_r':S^{4n+4r+3}\to X_{r-1}'$.
There is a natural map $\chi_r:X_r'\to X_r$ such that the diagram
	\begin{align*}
		\xymatrix{
			S^{4n+4r+3} \ar[r]^-{\varphi_r'} \ar[d]_-{p} & X_{r-1}' \ar[d]^-{\chi_r} & \mathbb{H}P^\infty_{(p)} \ar@{=}[d] \ar[l]_-{j_r'} \\
			S^{4n+4r+3} \ar[r]^-{\varphi_r} & X_{r-1} & \mathbb{H}P^\infty_{(p)} \ar[l]_-{j_r}
		}
	\end{align*}
commutes, which maps the characteristic map $\bar{\varphi}_r'\in\pi_{4n+4r+4}(X_r',X_{r-1}')$ of the cell $(e')^{4n+4r+4}\subset X_r'$ to $p\bar{\varphi}_r\in\pi_{4n+4r+4}(X_r,X_{r-1})$ and hence $\theta_r'=[(j_r')_*i,\bar{\varphi}_r']\in\pi_{4n+4r+7}(X_r')$ to $p\theta_r\in\pi_{4n+4r+7}(X_r)$.
Note the equality
	\begin{align*}
		\pi_{4n+4r+7}(X'_r)/(j'_r)_*\pi_{4n+4r+7}(\mathbb{H}P^\infty_{(p)})=\Z_{(p)}\{\varphi'_{r+1}\}=\Z_{(p)}\{\theta'_r\}.
	\end{align*}
Then there is a unit $c\in\Z_{(p)}^\times$ such that $\varphi_{r+1}'-c\theta_r'\in (j_r')_*\pi_{4n+4r+7}(\mathbb{H}P^\infty_{(p)})$.
From this, we have
\begin{align*}
	(\chi_r)_*(\varphi_{r+1}'-c\theta_r')=p\varphi_{r+1}-cp\theta_r=0
\end{align*}
because $\pi_{4n+4r+7}(\mathbb{H}P^\infty_{(p)})$ is annihilated by $p$ (the result of Selick \cite{Sel78}).
The injectivity of $(j_r)_*$ implies $\varphi_{r+1}'=c\theta_r'$ in $\pi_{4n+4r+7}(X_r')$.

Using the above computations, we shall construct a retraction $\rho_r:X_r'\to\mathbb{H}P^\infty_{(p)}$ by an induction on $r\le(p-1)/2$.
When $r=1$, the attaching map of $(e')^{4n+4r+8}\subset X_1'$ is null-homotopic as $\varphi_1'=p\varphi_1=0\in\pi_{4n+4r+7}(\mathbb{H}P^\infty_{(p)})$.
Thus a retraction $\rho_1:X_1'\to\mathbb{H}P^\infty_{(p)}$ exists.
Assume that there is a retraction $\rho_r:X_r'\to\mathbb{H}P^\infty_{(p)}$ for some $r<(p-1)/2$.
Then we have $(\rho_r)_*(\bar{\varphi}_r')\in\pi_{4n+4r+4}(\mathbb{H}P^\infty_{(p)},\mathbb{H}P^\infty_{(p)})=0$ and
	\begin{align*}
		(\rho_r)_*(\varphi_{r+1}')=(\rho_r)_*(c\theta_r')=c(\rho_r)_*([(j_r')_*i,\bar{\varphi}_r'])=c[i,(\rho_r)_*(\bar{\varphi}_r')]=0.
	\end{align*}
Hence there is a retraction $\rho_{r+1}:X_{r+1}'\to\mathbb{H}P^\infty_{(p)}$.

Therefore, the composition $\rho_r\circ f_r':S^4_{(p)}\times\mathbb{H}P^\infty_{(p)}\to\mathbb{H}P^\infty_{(p)}$ is an extension of $f\circ(p\times\mathrm{id})$.
\end{proof}
In the proof of the next corollary, we quote the result of the author \cite{Tsu12b}.
The author computed the divisibility of the invariant $\epsilon_n\in\mathbb{Q}$ defined by Tsukuda \cite{Tsu01}, which has the following property: if $\ad P_{k,p}$ is $A_n$-trivial, then $k\epsilon_n\in\mathbb{Z}_{(p)}$.
But the author must apologize for that the description of the main result \cite[Theorem4.5]{Tsu12b} is \textbf{not correct} because of a very elementary error of counting.
Nevertheless, the essential computation \cite[Proposition 4.2, 4.4]{Tsu12b} is correct.
\begin{cor}\label{odddivisibility}
For the integer $a_n^{\mathrm{fw}}$ defined above and an odd prime $p$, we have the $p$-exponent as
\begin{align*}
	v_p(a_n^{\mathrm{fw}})=\left[\frac{2n}{p-1}\right],
\end{align*}
where $[m]$ denotes the largest integer not greater than $m$.
\end{cor}
\begin{proof}
Take integers $q$ and $0\leq r<(p-1)/2$ as $n=q(p-1)/2+r$, where $q=[2n/(p-1)]$.
By \cite[Proposition 4.4]{Tsu12b}, we have $v_p(\epsilon_{q(p-1)/2})=-q$.
Since $v_p(a_n^{\mathrm{fw}}\epsilon_{q(p-1)/2})\geq 0$, we obtain
\begin{align*}
	v_p(a_n^{\mathrm{fw}})\geq q.
\end{align*}
By Proposition \ref{oddtriviality} (1), $\ad P_{1,p}$ is $A_r$-trivial.
Then, by using the Proposition \ref{oddtriviality} (2) repeatedly, $\ad P_{p^q,p}$ is $A_n$-trivial.
Therefore, by the definition of $a_n^{\mathrm{fw}}$, we have
\begin{align*}
	v_p(a_n^{\mathrm{fw}})\leq q.
\end{align*}
Finally, we conclude $v_p(a_n^{\mathrm{fw}})=q=[2n/(p-1)]$.
\end{proof}

\begin{rem}\label{2divisibility}
For $p=2$, by a similar observation of the obstruction to extending the map $S^4_{(2)}\vee\mathbb{H}P^n_{(2)}\to\mathbb{H}P^\infty_{(2)}$, we obtain
\begin{align*}
	n\leq v_2(a_n^{\mathrm{fw}})\leq 2n.
\end{align*}
The left hand side follows from \cite[Proposition 4.2]{Tsu12b}.
The right hand side follows from the fact that $4$ annihilates the $2$-component of the homotopy groups of $S^3$ by the result of James \cite[Corollay (1.22)]{Jam57}.
\end{rem}

We close this paper by proving Theorem \ref{SU(2)}.
\begin{proof}[Proof of Theorem \ref{SU(2)}]
The if part is a consequence of Corollay \ref{odddivisibility}, Remark \ref{2divisibility} and Theorem \ref{thm}.
The converse is shown as follows.
Suppose that the gauge groups $\mathcal{G}(P_k)$ and $\mathcal{G}(P_{k'})$ are $A_n$-equivalent and $v_2(k)\leq 1$.
By \cite[Proposition 10.2]{Tsu12a} and Corollay \ref{odddivisibility}, the equality $v_p(k)=v_p(k')$ must hold for any odd prime $p$.
Moreover, $v_2(k)=v_2(k')$ follows from the Kono's result \cite{Kon91} and $v_2(k)\leq 1$.
\end{proof}

\end{document}